\documentclass[11pt]{article}
\usepackage{amssymb, amsmath}

\usepackage[margin=1.2in]{geometry}

\usepackage{mathptmx}

\newtheorem{theorem}{Theorem}[section]
\newtheorem{lemma}[theorem]{Lemma}

\newtheorem{corollary}[theorem]{Corollary}
\newtheorem{conjecture}[theorem]{Conjecture}

\newenvironment{proof}[1][Proof]{\begin{trivlist}
\item[\hskip \labelsep {\bfseries #1}]}{\end{trivlist}}

\newenvironment{remark}[1][Remark]{\begin{trivlist}
\item[\hskip \labelsep {\bfseries #1}]}{\end{trivlist}}

\newcommand{\qed}{\nobreak \ifvmode \relax \else
      \ifdim\lastskip<1.5em \hskip-\lastskip
      \hskip1.5em plus0em minus0.5em \fi \nobreak
      \vrule height0.75em width0.5em depth0.25em\fi}

\renewcommand{\Re}{\operatorname{Re}}
\renewcommand{\Im}{\operatorname{Im}}

\begin{document}
\title{On Riemann zeroes, Lognormal Multiplicative Chaos, and Selberg Integral}

\author{Dmitry Ostrovsky}

\date{} 

\maketitle
\noindent

\begin{abstract}
\noindent Rescaled Mellin-type transforms of the exponential functional of 
the Bourgade-Kuan-Rodgers statistic of Riemann zeroes are conjecturally related
to the distribution of the total mass of the limit lognormal stochastic measure of Mandelbrot-Bacry-Muzy. The conjecture implies that a non-trivial, log-infinitely divisible probability distribution is associated with Riemann zeroes. For application, integral moments, covariance structure, multiscaling spectrum, and asymptotics associated with the exponential functional are computed in closed form using the known meromorphic extension of the Selberg integral. 
\end{abstract}

\noindent {\bf Keywords:} Riemann zeroes, multiplicative chaos, Selberg integral, multifractal stochastic measure, gaussian free field, infinite divisibility, double gamma function.

\noindent {\bf Mathematics Subject Classification (2010):} 11M26, 33B15, 60E07, 60G15, 60G57.

\section{Introduction}
\noindent In this paper we contribute to the literature on the statistical distribution of Riemann zeroes in the mesoscopic regime. The study of the values of the Riemann zeta function in the mesoscopic regime was pioneered by Selberg \cite{Selberg44}, \cite{Selberg46} and then extended to the distribution of zeroes by Fujii \cite{Fujii}, Hughes and Rudnick \cite{HR}, Bourgade \cite{Bourgade10}, and most recently by Bourgade and Kuan \cite{BK}, Rodgers \cite{Rodg}, and Kargin \cite{Kar}.
Assuming the Riemann hypothesis (except for Selberg), these authors rigorously established various central limit theorems for the distribution of Riemann zeroes. 
The principal technical tools that were used to obtain these theorems were Selberg's formula for $\zeta'/\zeta,$ explicit formulas of Guinand and Weil, and certain moment calculations. Alternatively, beginning with the seminal work of Montgomery \cite{Mont}, a great deal of progress has been made in formulating precise conjectures
about the statistical distribution of the zeroes. These conjectures are all motivated by the empirical fact that the statistical properties of the zeroes are very close to
those of eigenvalues of large Hermitian matrices with independent entries, \emph{i.e.} the so-called GUE matrices, up to small arithmetic corrections,
and calculations are typically justified by means of semi-classical methods for quantum chaotic systems, Keating-Snaith philosophy of modeling the value distribution
of the Riemann zeta function on the critical line by the characteristic polynomials of certain large random matrices, and conjectural forms of the approximate 
functional equation for the zeta function. 
For example, Berry \cite{Berry} calculated the GUE term and arithmetic corrections for the number variance, Bogomolny and Keating \cite{BoKe} did the same for the pair correlation function, which was later extended to multiple-point correlations by Conrey and Snaith \cite{CS} and Bogomolny and Keating \cite{BoKe2}, 
Keating and Snaith \cite{KS} calculated the moments of the zeta function on the critical line, Conrey \emph{et. al.} \cite{Conreyetal} formulated the ratios conjecture for its
average values, Farmer \emph{et. al.} \cite{Fetal} and Fyodorov and Keating \cite{YK} estimated the magnitude of its extreme values on the critical line, to name a few.

In this paper we conjecture a mod-gaussian limit theorem associated with the distribution of Riemann zeroes in the mesoscopic regime by combining the approach of Bourgade and Kuan with our previous work on the limit lognormal stochastic measure (also known as lognormal multiplicative chaos) and Selberg integral.
Bourgade and Kuan and Rodgers independently proved
that a class of linear statistics of Riemann zeroes converge to gaussian vectors and, most importantly, computed 
the covariance of the limiting vector explicitly. The starting point of our approach is that this limiting gaussian vector approximates the centered gaussian free field when the statistic is based on a smoothed indicator function of subintervals of the unit interval, and
it approximates the gaussian free field plus an independent gaussian random variable when the statistic is based on a smoothed indicator function of subintervals of a particular
unbounded interval. The limit lognormal measure is defined as a limit of the exponential functional of the gaussian free field, 
hence by taking the Mellin transform\footnote{It is more natural to define the Mellin transform as $\int_0^\infty
x^q\,f(x)\,dx$ as opposed to the usual $\int_0^\infty x^{q-1}\,f(x)\,dx$ for our purposes.} of 
the exponential functional of the Bourgade-Kuan-Rodgers statistic in an appropriately rescaled limit, we conjecturally 
obtain the Mellin transform of the total mass of the measure. In a series of papers, cf. \cite{Me2}, \cite{Me3}, \cite{Me4}, \cite{MeIMRN}, we investigated the total mass and made a precise conjecture about its probability distribution.\footnote{The terms
``probability distribution'', ``law'', and ``random variable'' are used
interchangeably in this paper.} The positive integral moments of the total mass are known to be given by the classical Selberg integral. In \cite{Me4} we rigorously constructed a probability distribution
(called the Selberg integral distribution) whose $n$th moment coincides with the Selberg integral of dimension $n$ and conjectured that this distribution is the same as the distribution of the total mass. Thus, the main result of this paper is a conjecture that  particular rescaled limits of two Mellin-type transforms of the exponential functional of the Bourgade-Kuan-Rodgers statistic corresponding to a smoothed indicator function of certain bounded or unbounded intervals 
coincide with the Mellin transform of the Selberg integral distribution. In \cite{MeIMRN}, \cite{Me13},
\cite{Me14} we established many properties of the Mellin transform so that we can make a number of precise statements about our rescaled limit as corollaries of the main conjecture. The type of rescaling and convergence that we use in this paper is closely related to the rescaling that was used by Keating and Snaith \cite{KS} and 
Nikeghbali and Yor \cite{NikYor}, formalized by Jacod \emph{et. al.} \cite{Jacod} in their theory of mod-gaussian convergence, and significantly extended in recent publications
of Feray \emph{et. al.} \cite{Feray} and M\'eliot and Nikeghbali \cite{Meliot}.

The main technical innovation of our work is the explicit use of the gaussian free field, limit lognormal measure, Selberg integral, and Selberg integral distribution in the context of the statistical distribution of Riemann zeroes. We note that the Selberg integral and Selberg integral distribution previously
appeared, respectively, in conjectures of Keating and Snaith \cite{KS} about the moments and of Fyodorov and Keating \cite{YK} about extreme values of the Riemann zeta function on the critical line. These conjectures are based on the analogy between the value distribution of $\zeta(1/2+it)$ and that of the characteristic polynomials of certain large random matrices. Our conjecture deals instead with the zeroes of the zeta function on the critical line and is based on the 
convergence of particular statistics of the zeroes to the gaussian free field or its centered version.  In particular, we prove that our statistics or, equivalently, particular integrals of $\Im\log \zeta(1/2+it)$ along the critical line, exhibit logarithmic correlations and calculate the corresponding covariances explicitly. We believe that these calculations are new.\footnote{The idea that $\Im\log \zeta(1/2+it)$ is logarithmically correlated is not new. Keating and Snaith \cite{KS} and later Farmer \emph{et. al.} \cite{Fetal} argued that $\Im\log \zeta(1/2+it)$ can be modeled
by the imaginary part of the logarithm of the characteristic polynomial of CUE matrices, and Hughes \emph{et. al.} \cite{Hughesetal} proved that the latter
is logarithmically correlated, thereby conjecturing the same about $\Im\log \zeta(1/2+it).$ The novelty of our work is the computation of 
the logarithmic covariance structure of $\Im\log \zeta(1/2+it)$ from first principles.}

The limit lognormal measure was introduced and reviewed by Mandelbrot \cite{secondface}, \cite{Lan} in the context of mathematical modeling of intermittent turbulence and constructed explicitly by Bacry and Muzy \cite{MRW},
\cite{BDM}, \cite{BM1}, \cite{BM}. Its existence and basic properties follow from the general theory of multiplicative chaos of Kahane \cite{K2}. This measure is of significant interest in mathematical physics as it naturally appears in a wide spectrum of problems ranging from conformal field theory \cite{BenSch}, \cite{RV1} and two-dimensional quantum gravity \cite{DS}, to 
statistical mechanics of disordered energy landscapes \cite{YO}, \cite{FLDR},  \cite{FLDR2}, \cite{YK}, to name a few. A periodized version of the limit lognormal measure appears in a random energy model \cite{FyoBou} and in the theory of conformal weldings 
\cite{Jones}. We also mention a multidimensional extension of the measure \cite{RV2} and a recent construction of the critical lognormal multiplicative chaos in \cite{dupluntieratal} and \cite{barraletal}. 
One of the most remarkable properties of this measure is that it is stochastically self-similar with lognormal multipliers (hence its name), so that the moments of its total mass do not scale linearly but rather quadratically, \emph{i.e.} the measure is multifractal. 
The aforementioned problems in mathematical physics all exhibit multifractal behavior so that the significance of our conjecture extends beyond the distribution of Riemann zeroes \emph{per se} for it suggests that the phenomenon of multifractality might have a number theoretic origin in the sense that the distribution of Riemann zeroes (conjecturally) provides a natural model for such phenomena. 

We do not have a mathematically rigorous proof of our conjecture and provide instead some exact calculations (a "physicist's proof") that explain how we arrived at it. If one assumes the conjecture to be true, the resulting corollaries are mathematically rigorous and their proofs can be found in \cite{MeIMRN} and \cite{Me14}.

The plan of this paper is as follows. In Section 2 we give a brief review of 
the key results of Bourgade and Kuan and Rodgers and then state our conjecture and its implications.
This section does  not require any knowledge of the limit lognormal measure.
In Section 3 we review the limit lognormal measure and the Selberg integral distribution. In Section 4 we present a heuristic derivation of our conjecture.
Conclusions are given in Section 5.

\section{Results}
\noindent We begin this section with a brief description of the statistic of Riemann zeroes that was introduced by Bourgade and Kuan \cite{BK} and Rodgers \cite{Rodg} (henceforth referred to as the BKR statistic), following the 
approach and notations of Bourgade and Kuan. The Riemann zeta function is defined\footnote{We will use the symbol $\triangleq$ to mean that the left-hand side is defined by the right-hand side.} by
\begin{equation}
\zeta(s) \triangleq \sum\limits_{m=0}^\infty (m+1)^{-s}, \, \Re(s)>1,
\end{equation}
and is continued analytically to the complex plane having a simple pole at $s=1.$ Its non-trivial zeroes are known to be located in the critical strip $0<\Re(s)<1$ and, according to the Riemann hypothesis, are thought to lie on the critical line $\Re(s)=1/2,$ cf. \cite{Titchmarsh} for details. 
Assuming the Riemann hypothesis, we write non-trivial zeroes of the Riemann zeta function in the form $\{1/2+i\gamma\},$
$\gamma\in\mathbb{R}.$ Let $\lambda(t)$ be a function of $t>0$ that satisfies
the asymptotic condition
\begin{equation}\label{lt}
1\ll \lambda(t) \ll \log t
\end{equation}
in the limit $t\rightarrow \infty,$ where the number theoretic notation $a(t)\ll b(t)$ means  $a(t)=o\bigl(b(t)\bigr).$
Let $\omega$ denote a uniform random variable over $(1, 2),$ $\gamma(t) \triangleq \lambda(t) (\gamma-\omega t),$ 
and define the statistic
\begin{equation}\label{St}
S_t(f) \triangleq \sum\limits_{\gamma} f\bigl(\gamma(t)\bigr) -\frac{\log t}{2\pi \lambda(t)} \int f(u) du
\end{equation}
given a test function $f(x).$ The class of test functions $\mathcal{H}^{1/2}$ that was considered in \cite{BK} is primarily defined by the condition $\langle f, \,f\rangle<\infty,$ where 
\begin{align}
\langle f,\,g\rangle \triangleq & \Re\int |w| \hat{f}(w)\overline{\hat{g}(w)}\,dw, \label{scalarf} \\
  = & -\frac{1}{2\pi^2} \int f'(x) g'(y)\log|x-y|dx\,dy \label{scalar}
\end{align}
plus some mild conditions on the growth of $f(x)$ and its Fourier transform 
$\hat{f}(w) \triangleq 1/2\pi \int f(x) e^{-iwx} \, dx$
at infinity and a bounded variation condition (that are satisfied by compactly supported $\mathcal{C}^2$ functions, by example). We note that $S_t(f)$ is centered\footnote{Centered means that its expectation is zero. All expectations, covariances, etc in this section are with respect to the distribution of $\omega.$}
in the limit $t\rightarrow \infty$ as it is well known that the number of Riemann zeroes in the interval $[t,\,2t]$ is asymptotic to $t\log t/2\pi$ in this limit. The principal
results of \cite{BK} and \cite{Rodg}\footnote{Rodgers considered a more restrictive class of test functions and stated the formula for the variance only.}
and the starting point of our construction are the following theorems.
\begin{theorem}[Convergence to a gaussian vector]\label{strong}
Given test function $f_1,$ $\cdots$ $f_k$ in $\mathcal{H}^{1/2},$
the random vector $\bigl(S_t(f_1),\cdots, S_t(f_k)\bigr)$ converges in law
in the limit $t\rightarrow\infty$ to
a centered gaussian vector $\bigl(S(f_1),\cdots S(f_k)\bigr)$ having the covariance 
\begin{equation}
{\bf Cov} \bigl(S(f_i), \,S(f_j)\bigr) = \langle f_i,\,f_j\rangle.
\end{equation}
\end{theorem}
The second theorem deals with the case of diverging limiting variance.
Define
\begin{equation}\label{limitvar}
\sigma_t(f)^2 \triangleq \int\limits_{-\log t/\lambda(t)}^{\log t/\lambda(t)} |w| |\hat{f}(w)|^2\,dw,
\end{equation}
then, under the assumption that $\sigma_t(f)\rightarrow \infty$ as $t\rightarrow \infty,$ 
\begin{theorem}[Convergence in the case of diverging variance]\label{divergvar}
\begin{equation}
S_t(f)/\sigma_t(f) \overset{{\rm in \,law}}{\longrightarrow} \mathcal{N}(0, 1),
\end{equation}
where $\mathcal{N}(0, 1)$ denotes the standard gaussian random variable with the zero mean and unit variance.
\end{theorem}
The significance of the condition $\lambda(t)\ll \log t$ is that the number of zeroes that are visited by $f$ as $t\rightarrow \infty$ goes to infinity, \emph{i.e.} Theorems \ref{strong} and \ref{divergvar}  are mesoscopic central limit theorems. 

The intuitive meaning of the BKR theorems and the statistic $S_t(f)$ can be established from the connection of $S_t(f)$ with the error 
term $S(t)$ in the zero counting function $N(t).$ Let $N(t)$ denote
the number of Riemann zeroes having their imaginary part (``height'') between zero and $t.$ Let the function $S(t)$ be defined by
\begin{equation}
S(t) \triangleq \frac{1}{\pi}\arg \zeta(1/2+it) = \frac{1}{\pi}\Im \log\zeta(1/2+it).
\end{equation}
Then, the Riemann-von Mangoldt formula states
\begin{align}
N(t) & =  \frac{1}{\pi} \Bigl(\Im \log\Gamma(\frac{1}{4}+\frac{it}{2}) -\frac{t}{2}\log\pi\Bigr)+1 + S(t),\\
& =  \frac{t}{2\pi} \log \frac{t}{2\pi} -\frac{t}{2\pi} + S(t) + \frac{7}{8} + O(1/t) \label{RMasym}
\end{align}
in the limit $t\rightarrow \infty.$ Let $u$ be fixed and $\lambda(t)$ be as in \eqref{lt}, then the asymptotic in equation \eqref{RMasym} implies that the number of zeroes 
in the random interval $[\omega t, \omega t+u/\lambda(t)]$ satisfies in the same limit
\begin{equation}
N\Bigl(\omega t+\frac{u}{\lambda(t)}\Bigr)-N(\omega t) = u\, \frac{\log t}{2\pi\lambda(t)} + S\Bigl(\omega t+\frac{u}{\lambda(t)}\Bigr)-S(\omega t) +O\bigl(1/\lambda(t)\bigr).\label{numinterval}
\end{equation}
It follows that the expected number of zeroes in $[\omega t, \omega t+u/\lambda(t)]$ is given by the leading asymptotic $u\log t/2\pi\lambda(t) ,$
whereas $S\bigl(\omega t+u/\lambda(t)\bigr)-S(\omega t)$ gives ``the error'', \emph{i.e.} the fluctuation of the number of zeroes in this interval from its leading
asymptotic. Note that the condition in \eqref{lt} means that the length of the interval
goes to zero, whereas the expected number of zeroes goes to infinity in the limit $t\rightarrow \infty,$ \emph{i.e.} the interval is mesoscopic.
Let $\chi_u(x)$ denote the indicator function of the interval $[0, \,u],$ then the corresponding BKR statistic satisfies by \eqref{St} and
\eqref{numinterval}
\begin{align}
S_t(\chi_u) & = N\Bigl(\omega t+\frac{u}{\lambda(t)}\Bigr)-N(\omega t) - u \, \frac{\log t}{2\pi\lambda(t)}, \\
& = S\Bigl(\omega t+\frac{u}{\lambda(t)}\Bigr)-S(\omega t) + O\bigl(1/\lambda(t)\bigr).\label{StS}
\end{align}
Hence the statistics measures the fluctuation of the error term over the random interval $[\omega t, \omega t+u/\lambda(t)].$
It is easy to see from \eqref{limitvar} that the corresponding variance has the leading asymptotic $\sigma_t^2(\chi_u) \approx \log \bigl(\log t/\lambda(t)\bigr)/\pi^2$
so that Theorem \ref{divergvar} gives us
\begin{equation}\label{FujiiCLT}
\frac{S\bigl(\omega t+u/\lambda(t)\bigr)-S(\omega t)}{\sqrt{\log \bigl(\log t/\lambda(t)\bigr)}/\pi} \overset{{\rm in \,law}}{\longrightarrow}  \mathcal{N}(0, 1).
\end{equation}
This special case of Theorem \ref{divergvar} is known as Fujii's central limit theorem, cf. \cite{Fujii}, \cite{Fujii2}. It turns out that the interpretation of the BKR
statistic as a measure of fluctuation of the error term remains true in general. We have the following identity for compactly supported test functions\footnote{
This equation does not appear in \cite{BK} but follows easily from intermediate steps in their derivation of Proposition 3. 
In all of our applications, cf. \eqref{Su} and \eqref{Sut} below, the constant in $O\bigl(1/\lambda(t)\bigr)$ is of the order $O\bigl(||f||_1\bigr)$ so
that the $O\bigl(1/\lambda(t)\bigr)$ term is negligible for our purposes due to $||f||_1/\lambda(t)\ll 1.$ A proof of \eqref{resultS} is given in the appendix.
}
\begin{equation}\label{resultS}
S_t(f) = - \int f'(x) \,S\Bigl(\omega t + \frac{x}{\lambda(t)}\Bigr)\,dx + O\bigl(1/\lambda(t)\bigr).
\end{equation}
Clearly, this equation recovers \eqref{StS} in the case of $f=\chi_u$ since $f'$ is simply the difference of the delta functions at the endpoints.

We now proceed to state our results. Let $0<u<1$ and $\chi^{(i)}_u(x),$ $i=1,2,$ denote the indicator functions of the intervals $[0, \,u]$
and $[-1/\varepsilon, \,u]$ for some fixed $\varepsilon>0,$ respectively. Let $\phi(x)$ be a smooth bump function supported on $(-1/2, \,1/2),$ 
and denote
\begin{equation}\label{kappa}
\kappa\triangleq -\int
\phi(x)\phi(y)\log|x-y|\,dxdy. 
\end{equation}
Define the $\varepsilon$-rescaled bump function by $\phi_\varepsilon(x)\triangleq 1/\varepsilon\phi(x/\varepsilon),$
and let $f^{(i)}_{\varepsilon, u}(x)$ be the smoothed indicator functions of 
the intervals $[0, \,u]$ and $[-1/\varepsilon, \,u]$ given by 
the convolution of $\chi^{(i)}_u(x)$ with $\phi_\varepsilon(x)$
\begin{equation}\label{fu}
f^{(i)}_{\varepsilon, u}(x) \triangleq (\chi^{(i)}_u\star\phi_\varepsilon)(x) = \frac{1}{\varepsilon}\int \chi^{(i)}_u(x-y)\phi(y/\varepsilon)\,dy.
\end{equation}
Clearly, $f^{(1)}_{\varepsilon, u}(x)=1$ for $x\in[\varepsilon/2,\,u-\varepsilon/2],$
$f^{(2)}_{\varepsilon, u}(x)=1$ for $x\in[-1/\varepsilon+\varepsilon/2,\,u-\varepsilon/2],$
$f^{(1)}_{\varepsilon, u}(x)=0$ for $x\geq u+\varepsilon/2$ and $x\leq -\varepsilon/2,$ 
and $f^{(2)}_{\varepsilon, u}(x)=0$ for $x\geq u+\varepsilon/2$ and $x\leq -1/\varepsilon-\varepsilon/2.$ 
Moreover, 
\begin{align}
f'^{(1)}_{\varepsilon, u}(x) = & \phi_\varepsilon(x) - \phi_\varepsilon(x-u), \label{fuprime1} \\
f'^{(2)}_{\varepsilon, u}(x) = & \phi_\varepsilon(x+1/\varepsilon) - \phi_\varepsilon(x-u). \label{fuprime2} 
\end{align}
Theorem \ref{strong} applies to $f^{(i)}_{\varepsilon, u}(x)$ for all $u>\varepsilon>0.$ Fix $\varepsilon>0$ and denote the BKR statistic based on some constant\footnote{The choice of the constant $\pi\sqrt{2\mu},$ $0<\mu<2$ will be explained in Section 4, cf. Corollary \ref{covstructureS}.} times $f^{(i)}_{\varepsilon, u}(x)$ by $S^{(i)}_t(\mu,u,\varepsilon)$
\begin{equation}\label{Su}
S^{(i)}_t(\mu, u, \varepsilon)\triangleq \pi\sqrt{2\mu}\Bigl[\sum\limits_{\gamma} (\chi^{(i)}_u\star\phi_{\varepsilon})\bigl(\gamma(t)\bigr) - \frac{\log t}{2\pi \lambda(t)} \int (\chi^{(i)}_u\star\phi_{\varepsilon})(x) dx\Bigr].
\end{equation}
The meaning of the first statistic $(i=1)$ is that it counts zeroes in the interval
$[t, 2t]$ over three asymptotic scales. Specifically,
over the interval of length $u/\lambda(t)-\varepsilon/\lambda(t),$
\begin{equation}\label{middle1}
\gamma-\frac{u}{\lambda(t)}+\frac{\varepsilon}{2\lambda(t)} < \omega t <
\gamma - \frac{\varepsilon}{2\lambda(t)},
\end{equation}
the zero is counted with the weight 1, whereas it is counted with a diminishing
weight that is determined by $\varepsilon$ and the bump function over the boundary intervals of length $\varepsilon/\lambda(t),$
\begin{align}
\gamma-\frac{\varepsilon}{2\lambda(t)} <& \omega t <
\gamma + \frac{\varepsilon}{2\lambda(t)}, \label{rboundary1}\\
\gamma-\frac{u}{\lambda(t)}-\frac{\varepsilon}{2\lambda(t)} <& \omega t <
\gamma - \frac{u}{\lambda(t)}+ \frac{\varepsilon}{2\lambda(t)}.
\end{align}
The second statistic $(i=2)$ has the same interpretation except that instead of \eqref{middle1} and \eqref{rboundary1} we have
\begin{align}
\gamma-\frac{u}{\lambda(t)}+\frac{\varepsilon}{2\lambda(t)} <& \omega t <
\gamma - \frac{\varepsilon}{2\lambda(t)} + \frac{1}{\varepsilon\lambda(t)}, \\
\gamma-\frac{\varepsilon}{2\lambda(t)} +\frac{1}{\varepsilon\lambda(t)} <& \omega t <
\gamma + \frac{\varepsilon}{2\lambda(t)} +\frac{1}{\varepsilon\lambda(t)},
\end{align}
respectively, so that  the middle interval 
has length $u/\lambda(t)-\varepsilon/\lambda(t)+1/\varepsilon\lambda(t) ,$ whereas the boundary intervals have length $\varepsilon/\lambda(t)$ as before.
Hence, the three scales are $\varepsilon/\lambda(t),$ $u/\lambda(t),$ $t,$
and they satisfy the asymptotic condition
\begin{equation}
\frac{1}{\log t}\ll \frac{\varepsilon}{\lambda(t)} < \frac{u}{\lambda(t)}\ll 1\ll t.
\end{equation}
This means that $t$ defines the global scale and $u/\lambda(t)\gg \text{average spacing}$ so that $u/\lambda(t)$ 
and $\varepsilon/\lambda(t)$ are on the mesoscopic scale.

The meaning of the statistics $S^{(i)}_t(\mu, u, \varepsilon)$ can be elucidated further by means of Theorem \ref{strong} and \eqref{resultS}.
Given the expressions for the derivatives in \eqref{fuprime1} and \eqref{fuprime2}, which show that the derivatives are smooth approximations 
of the difference of the delta functions at the endpoints, we can write
\begin{align}
S^{(1)}_t(\mu, u, \varepsilon) & = \pi\sqrt{2\mu} \int \bigl(\phi_\varepsilon(x-u) - \phi_\varepsilon(x) \bigr)\,
S\Bigl(\omega t + \frac{x}{\lambda(t)}\Bigr)\,dx + O\bigl(1/\lambda(t)\bigr), \label{ourstats1}\\
S^{(2)}_t(\mu, u, \varepsilon) & = \pi\sqrt{2\mu} \int \bigl(\phi_\varepsilon(x-u)-\phi_\varepsilon(x+1/\varepsilon)\bigr) \,
S\Bigl(\omega t + \frac{x}{\lambda(t)}\Bigr)\,dx +O\bigl(1/\lambda(t)\bigr),\label{ourstats2}
\end{align}
so that the statistics are smooth approximations of 
\begin{align}
 & \pi\sqrt{2\mu} \Bigl[S\Bigl(\omega t + \frac{u}{\lambda(t)}\Bigr) - S(\omega t)\Bigr], \\
 & \pi\sqrt{2\mu} \Bigl[S\Bigl(\omega t + \frac{u}{\lambda(t)}\Bigr) - S\Bigl(\omega t-\frac{1}{\varepsilon \lambda(t)} \Bigr)\Bigr],
\end{align}
and can be interpreted as fluctuations of the smoothed error term over $[\omega t,\,\omega t +u/\lambda(t)]$ and $[\omega t-1/\varepsilon\lambda(t),\,\omega t+u/\lambda(t)],$ respectively. 
Moreover, by Theorem \ref{strong} and calculations in Corollaries \ref{covstructureS} and \ref{covstructureS2},
the processes $u\rightarrow S^{(i)}_t(\mu, u, \varepsilon),$ $u\in (0, 1),$
converge in law in the limit $t\rightarrow\infty$  to centered gaussian fields having the asymptotic covariances,
for $i=1,$
\begin{align}\label{limcov1}
\begin{cases}
& -\mu 
\bigl(\log\varepsilon - \kappa + \log|u-v| - \log|u| - \log|v| \bigr) + O(\varepsilon), \; \text{if $|u-v|\gg \varepsilon$}, 
 \\
 &  -2\mu 
\bigl(\log\varepsilon - \kappa - \log|u| \bigr) + O(\varepsilon),
 \; \text{if $u=v,$}
\end{cases}
\end{align}
and, for $i=2,$
\begin{align}\label{limcov2}
\begin{cases}
 -\mu 
\bigl(3\log\varepsilon - \kappa + \log|u-v| \bigr) + O(\varepsilon), \; \text{if $|u-v|\gg \varepsilon$}, 
 \\
-\mu \bigl(4\log\varepsilon - 2\kappa  \bigr) + O(\varepsilon),
 \; \text{if $u=v,$}
\end{cases}
\end{align}
thereby exhibiting logarithmic correlations. These covariances
will be identified in Section 4 as corresponding to the centered gaussian free field and the gaussian free field plus an independent gaussian random variable,
respectively. We see from \eqref{limcov1} and \eqref{limcov2} that the asymptotic covariances
of both statistics diverge as $\log\varepsilon$ in the limit $\varepsilon\rightarrow 0.$ We are primarily interested in the statistical structure
of the $\log|u-v|$ terms, which is hiding behind these divergences. We will remove them so as to reveal the underlying structure by
introducing rescaling factors into the Mellin transforms as shown below. 

It should also be noted that the need for smoothing is necessitated by the singularity of our statistics in the limit
$\varepsilon\rightarrow 0.$ Indeed, $S^{(2)}_t(\mu, u, \varepsilon)$ is not even defined in this limit, and 
${\bf Var}[  S^{(1)}_t(\mu, u, 0)] \propto \log(\log t/\lambda(t))$
in Fujii's theorem, cf. \eqref{FujiiCLT}, is
asymptotically different from ${\bf Var}[\lim\limits_{t\rightarrow \infty}  S^{(1)}_t(\mu, u, \varepsilon)] \propto -\log(\varepsilon)$ in
\eqref{limcov1}. 
This is explained by the difference between the formulas for the variance in 
\eqref{scalarf} and \eqref{limitvar}, in fact, as shown in Lemma \ref{limit}, the difference between the two for 
smoothed indicator functions is of the order $O(\lambda(t)/\varepsilon \log t)$ and so becomes significant if 
one takes the $\varepsilon\rightarrow 0$ limit first. To bypass this singularity, in what follows we will always take
the $t$ limit first and, when the two limits are taken simultaneously, $\varepsilon(t)$ will vary ``slowly''
enough to achieve he same end.

Motivated by the proof of Theorem \ref{divergvar} in \cite{BK}, we are also interested in the same type of statistics that 
are based on $t$-dependent $\varepsilon(t).$ 
Let
\begin{gather}
\frac{\lambda(t)}{\log t}\ll\varepsilon(t) \ll 1, \; \text{for}\;i=1,\label{vart}\\
\frac{\lambda(t)}{\log t}\ll\varepsilon(t) \ll 1 \; \text{and} \;\frac{1}{\lambda(t)}\ll\varepsilon(t),\; \text{for}\; i=2.\label{vart2}
\end{gather}
In other words, we replace 
$\varepsilon$ with $\varepsilon(t)$
in \eqref{Su}, resulting in the statistics\footnote{We note that $\chi^{(2)}_u(x) =\chi_{[-1/\varepsilon(t),\,u]}(x)$
becomes $t$-dependent so that $||\chi^{(2)}_u||_1=O\bigl(1/\varepsilon(t)\bigr)$ and the $O\bigl(1/\lambda(t)\bigr)$ term in \eqref{resultS} is negligible due to $\varepsilon(t)\lambda(t)\gg1$ in \eqref{vart2}.}
$S^{(i)}_t(\mu, u) \triangleq S^{(i)}_t(\mu, u, \varepsilon(t)),$
\begin{equation}\label{Sut}
S^{(i)}_t(\mu, u)\triangleq \pi\sqrt{2\mu}\Bigl[\sum\limits_{\gamma} (\chi^{(i)}_u\star\phi_{\varepsilon(t)})\bigl(\gamma(t)\bigr) - \frac{\log t}{2\pi \lambda(t)} \int (\chi^{(i)}_u\star\phi_{\varepsilon(t)})(x) dx\Bigr].
\end{equation}
These statistics also count zeroes in the interval
$[t, 2t]$ over the three scales $\varepsilon(t)/\lambda(t),$ $u/\lambda(t),$ $t,$ as before, except that $u/\lambda(t)\gg\varepsilon(t)/\lambda(t) \gg \text{average spacing},$
\begin{equation}
\frac{1}{\log t} \ll \frac{\varepsilon(t)}{\lambda(t)} \ll\frac{u}{\lambda(t)}\ll 1\ll t.
\end{equation}
\begin{remark}
We note that the asymptotic conditions in \eqref{vart} and \eqref{vart2} have a simple interpretation of being ``slow'' decay conditions
that require that the lengths of intervals, over which the zeroes are counted, go to zero, whereas the expected numbers of the zeroes counted go to infinity.
The lengths of the middle and boundary intervals of the first statistic satisfy $O\bigl(1/\lambda(t)\bigr)$ 
and $O\bigl(\varepsilon(t)/\lambda(t)\bigr),$ so that the corresponding expected numbers of zeroes are 
$O\bigl(\log t/\lambda(t)\bigr)$ and $O\bigl(\varepsilon(t)\,\log t/\lambda(t)\bigr)$ by \eqref{numinterval}, respectively. Thus, the lengths go to zero by \eqref{lt} and 
the expected numbers go to infinity by \eqref{lt} and \eqref{vart}. Similarly, the lengths of the middle and boundary intervals of the second statistic are
$O\bigl(1/\lambda(t)+1/\varepsilon(t)\,\lambda(t)\bigr)$ and $O\bigl(\varepsilon(t)/\lambda(t)\bigr)$ so that the above argument remains valid
provided $\varepsilon(t)\,\lambda(t)\gg 1,$ hence \eqref{vart2}. We finally note that the interpretation of $S^{(i)}_t(\mu, u)$ as fluctuations
of smoothed error terms and calculations of limiting covariances in \eqref{limcov1} and \eqref{limcov2} remain valid provided \eqref{vart} and \eqref{vart2}
are satisfied as shown in Lemmas \ref{modified} and \ref{limit}, \emph{i.e.} the limiting fields are gaussian and have $\varepsilon(t)$-dependent asymptotic covariance given in \eqref{limcov1} and \eqref{limcov2} with $\varepsilon=\varepsilon(t).$
\end{remark}

The statements of our results below require some familiarity with the double gamma function of Barnes (or the Alexeiewsky-Barnes $G-$function). We will give a brief summary of its definition and properties here and refer the reader to \cite{Genesis}, \cite{mBarnes}  for the original construction, to \cite{Ruij} for a modern treatment, and to \cite{MeIMRN} and \cite{Me14} for detailed reviews. One starts with the double zeta function
\begin{equation}
\zeta_2(s, w\,|\,\tau) \triangleq \sum\limits_{k_1, k_2=0}^\infty (w+k_1+k_2\tau)^{-s}
\end{equation}
that is defined for $\Re(s)>2,$ $\Re(w)>0,$ and $\tau>0.$ It can be analytically extended to $s\in\mathbb{C}$ with simple poles at $s=1$ and $s=2.$
The double gamma function is then defined as the exponential of the $s$-derivative of $\zeta_2(s, w|\tau)$ at $s=0,$
\begin{equation}
\Gamma_2(w\,|\,\tau) \triangleq \exp\Bigl(\partial_s\big\vert_{s=0} \zeta_2(s, w\,|\,\tau)\Bigr).
\end{equation}
The resulting function can be analytically extended to $w\in\mathbb{C}$ having no zeroes and poles at 
\begin{equation}
w=-(k_1+k_2\tau), \, k_1, k_2\in\mathbb{N}.
\end{equation}
Barnes gave an infinite product formula for $\Gamma_2(w\,|\,\tau),$
which in our normalization takes on the form
\begin{equation}
\Gamma^{-1}_2(w\,|\,\tau) = e^{P(w\,|\,\tau)}\,w\prod\limits_{k_1, k_2=0}^\infty
{}' \Bigl(1+\frac{w}{k_1+k_2\tau}\Bigr)\exp\left(-\frac{w}{k_1+k_2\tau}+
\frac{w^2}{2(k_1+k_2\tau)^2}
\right), 
\end{equation}
where $P(w\,|\,\tau)$ is a polynomial in $w$ of degree $2,$ and the prime indicates that the product is over all indices except $k_1=k_2=0.$
The double gamma function satisfies the functional equations
\begin{align}
\frac{\Gamma_2(z\,|\,\tau)}{\Gamma_2(z+1\,|\,\tau)}
& = \frac{\tau^{z/\tau-1/2}}
{\sqrt{2\pi}}\Gamma\Bigl(\frac{z}{\tau}\Bigr), \\
\frac{\Gamma_2(z\,|\,\tau)}{\Gamma_2(z+\tau\,|\,\tau)}
& = \frac{1}
{\sqrt{2\pi}}\Gamma\bigl(z\bigr),
\end{align}
where $\Gamma(z)$ denotes Euler's gamma function.
An explicit integral representation of $\log\Gamma_2(w\,|\,\tau)$ and additional
infinite product representations of $\Gamma_2(w\,|\,\tau)$ 
can be found in \cite{MeIMRN}.

From now on, it is always assumed that $\lambda(t)$ and $\varepsilon(t)$ satisfy \eqref{lt} and \eqref{vart} or \eqref{vart2}, respectively, and $\kappa$ is 
as in \eqref{kappa}. Given $0<\mu<2,$ henceforth let 
\begin{equation}
\tau\triangleq \frac{2}{\mu}>1.
\end{equation}
The results given below are stated for the $S^{(1)}_t(\mu,u,\varepsilon)$ and $S^{(1)}_t(\mu,u)$ statistics to avoid
redundancy, for the same formulas apply to $S^{(2)}_t(\mu,u,\varepsilon)$ provided one simultaneously replaces
\begin{align}
e^{S^{(1)}_t(\mu,u,\varepsilon)} &\longleftrightarrow u^{\mu q} e^{S^{(2)}_t(\mu,u,\varepsilon)}, \label{transl1}\\
e^{\mu(\log \varepsilon-\kappa)\frac{q(q+1)}{2}} &\longleftrightarrow e^{\mu\log\varepsilon q^2} e^{\mu(\log \varepsilon-\kappa)\frac{q(q+1)}{2}},\label{transl2}
\end{align}
and, similarly to $S^{(2)}_t(\mu,u),$ provided one replaces $e^{S^{(1)}_t(\mu,u)} \longleftrightarrow u^{\mu q} e^{S^{(2)}_t(\mu,u)}$ and uses \eqref{vart2} instead of \eqref{vart}.
For clarity, this translation is shown explicitly in Conjectures \ref{Mellinvar} and \ref{Mellin}. Given these preliminaries, 
our results are as follows. 

\begin{conjecture}[Mellin-type transforms: weak version]\label{Mellinvar} 
Let $\Re(q)<\tau,$ then
\begin{align}
&\lim\limits_{\varepsilon\rightarrow 0} e^{\mu(\log \varepsilon-\kappa)\frac{q(q+1)}{2}} \Bigl[\lim\limits_{t\rightarrow\infty} 
{\bf E} \Bigl[\Bigl(\int_0^1 u^{-\mu q} e^{S^{(1)}_t(\mu,u,\varepsilon)} du\Bigr)^q\Bigr]\Bigr], \label{M2transf}\\ 
&= \lim\limits_{\varepsilon\rightarrow 0}  e^{\mu\log\varepsilon q^2} e^{\mu(\log \varepsilon-\kappa)\frac{q(q+1)}{2}} \Bigl[\lim\limits_{t\rightarrow\infty} 
{\bf E} \Bigl[\Bigl(\int_0^1 e^{S^{(2)}_t(\mu,u,\varepsilon)} du\Bigr)^q\Bigr]\Bigr], \label{M2transfstrong}\\
& =
\tau^{\frac{q}{\tau}} (2\pi)^{q}\,\Gamma^{-q}\bigl(1-1/\tau\bigr)
\frac{\Gamma^2_2(1-q+\tau\,|\,\tau)}{\Gamma^2_2(1+\tau\,|\,\tau)}
\frac{\Gamma_2(-q+\tau\,|\,\tau)}{\Gamma_2(\tau\,|\,\tau)}
\frac{\Gamma_2(2-q+2\tau\,|\,\tau)}{\Gamma_2(2-2q+2\tau\,|\,\tau)}
\label{M2}
.
\end{align}
Let $-(\tau+1)/2<\Re(q)<\tau,$ then
\begin{align}
& \lim\limits_{\varepsilon\rightarrow 0} e^{\mu(\log \varepsilon-\kappa)\frac{q(q+1)}{2}} \Big[\lim\limits_{t\rightarrow\infty}  
{\bf E} \Bigl[\Bigl(\int_0^1 e^{S^{(1)}_t(\mu,u,\varepsilon)} du\Bigr)^q\Bigr]\Bigr], \label{M1transf}\\
& = \lim\limits_{\varepsilon\rightarrow 0}  e^{\mu\log\varepsilon q^2} e^{\mu(\log \varepsilon-\kappa)\frac{q(q+1)}{2}} \Bigl[\lim\limits_{t\rightarrow\infty} 
{\bf E} \Bigl[\Bigl(\int_0^1 u^{\mu q}e^{S^{(2)}_t(\mu,u,\varepsilon)} du\Bigr)^q\Bigr]\Bigr], \label{M1transfstrong}\\
& =  \tau^{\frac{q}{\tau}} (2\pi)^{q}\,\Gamma^{-q}\bigl(1-1/\tau\bigr)
\frac{\Gamma_2(1+q+\tau\,|\,\tau)}{\Gamma_2(1+2q+\tau\,|\,\tau)}
\frac{\Gamma_2(1-q+\tau\,|\,\tau)}{\Gamma_2(1+\tau\,|\,\tau)}
\frac{\Gamma_2(-q+\tau\,|\,\tau)}{\Gamma_2(\tau\,|\,\tau)}
\frac{\Gamma_2(2+q+2\tau\,|\,\tau)}{\Gamma_2(2+2\tau\,|\,\tau)}.
\label{M1}
\end{align}
\end{conjecture}
\begin{conjecture}[Mellin-type transforms: strong version]\label{Mellin} 
Let $\Re(q)<\tau,$ then
\begin{align}
\lim\limits_{t\rightarrow\infty} e^{\mu(\log \varepsilon(t)-\kappa)\frac{q(q+1)}{2}} 
{\bf E} \Bigl[\Bigl(\int_0^1 u^{-\mu q} e^{S^{(1)}_t(\mu,u)} du\Bigr)^q\Bigr] = &
\lim\limits_{\varepsilon\rightarrow 0} e^{\mu(\log \varepsilon-\kappa)\frac{q(q+1)}{2}} \star \nonumber \\ &\star\Bigl[\lim\limits_{t\rightarrow\infty} 
{\bf E} \Bigl[\Bigl(\int_0^1 u^{-\mu q} e^{S^{(1)}_t(\mu,u,\varepsilon)} du\Bigr)^q\Bigr]\Bigr], \label{M2v} \\
\lim\limits_{t\rightarrow\infty}  e^{\mu\log\varepsilon(t) q^2} e^{\mu(\log \varepsilon(t)-\kappa)\frac{q(q+1)}{2}} 
{\bf E} \Bigl[\Bigl(\int_0^1 e^{S^{(2)}_t(\mu,u)} du\Bigr)^q\Bigr] = &
\lim\limits_{\varepsilon\rightarrow 0}  e^{\mu\log\varepsilon q^2}e^{\mu(\log \varepsilon-\kappa)\frac{q(q+1)}{2}}\star\nonumber \\ &\star\Bigl[\lim\limits_{t\rightarrow\infty} 
{\bf E} \Bigl[\Bigl(\int_0^1 e^{S^{(2)}_t(\mu,u,\varepsilon)} du\Bigr)^q\Bigr]\Bigr]. \label{M2vstrong}
\end{align}
Let $-(\tau+1)/2<\Re(q)<\tau,$ then
\begin{equation}
\lim\limits_{t\rightarrow\infty} e^{\mu(\log \varepsilon(t)-\kappa)\frac{q(q+1)}{2}} 
{\bf E} \Bigl[\Bigl(\int_0^1 e^{S^{(1)}_t(\mu,u)} du\Bigr)^q\Bigr] = 
\lim\limits_{\varepsilon\rightarrow 0} e^{\mu(\log \varepsilon-\kappa)\frac{q(q+1)}{2}} \Bigl[\lim\limits_{t\rightarrow\infty} 
{\bf E} \Bigl[\Bigl(\int_0^1 e^{S^{(1)}_t(\mu,u,\varepsilon)} du\Bigr)^q\Bigr]\Bigr],
\label{M1v}
\end{equation}
\begin{align}
\lim\limits_{t\rightarrow\infty} e^{\mu\log\varepsilon(t) q^2} e^{\mu(\log \varepsilon(t)-\kappa)\frac{q(q+1)}{2}} 
{\bf E} \Bigl[\Bigl(\int_0^1 u^{\mu q}\,e^{S^{(2)}_t(\mu,u)} du\Bigr)^q\Bigr] = &
\lim\limits_{\varepsilon\rightarrow 0} e^{\mu\log\varepsilon q^2} e^{\mu(\log \varepsilon-\kappa)\frac{q(q+1)}{2}}\star
\nonumber \\ &\star \Bigl[\lim\limits_{t\rightarrow\infty} 
{\bf E} \Bigl[\Bigl(\int_0^1 u^{\mu q} \, e^{S^{(2)}_t(\mu,u,\varepsilon)} du\Bigr)^q\Bigr]\Bigr].
\label{M1vstrong}
\end{align}
\end{conjecture}
The following corollaries can be formulated with either the double or single  limits as in Conjectures \ref{Mellinvar} and \ref{Mellin}. We will state them with the single limit for notational simplicity. All corollaries apply to both statistics, cf. \eqref{transl1}
and \eqref{transl2} for the translation. Note that in Corollaries \ref{posmoments}, \ref{negmoments}, and
\ref{jointintegralmom} we have $q=n, \,-n, \,N,$ respectively, so that the multiplier $u^{\mu q}$ in \eqref{transl1} and the scaling factor 
in \eqref{transl2} are adjusted accordingly, when translated for the second statistic. 
\begin{corollary}[Positive integral moments and Selberg integral]\label{posmoments}
Let $n\in\mathbb{N}$ such that $n<\tau.$
\begin{align}
\lim\limits_{t\rightarrow\infty}  e^{\mu(\log \varepsilon(t)-\kappa)\frac{n(n+1)}{2}} 
{\bf E} \Bigl[\Bigl(\int_0^1 u^{-\mu n} e^{S^{(1)}_t(\mu,u)} du\Bigr)^n\Bigr] & = 
\int\limits_{[0,\,1]^n} \prod\limits_{i<j}^n |s_i-s_j|^{-\mu} ds_1\cdots ds_n,  \label{selberg2}\\
& = \prod_{k=0}^{n-1} \frac{\Gamma(1-(k+1)/\tau)}{\Gamma(1-1/\tau)}
\frac{\Gamma^2(1-k/\tau)}
{\Gamma(2-(n+k-1)/\tau)}, \label{mom2plus} \\
\lim\limits_{t\rightarrow\infty} e^{\mu(\log \varepsilon(t)-\kappa)\frac{n(n+1)}{2}}  
{\bf E} \Bigl[\Bigl(\int_0^1 e^{S^{(1)}_t(\mu,u)} du\Bigr)^n\Bigr] & = 
\int\limits_{[0,\,1]^n} \prod_{i=1}^n s_i^{\mu n}\, \prod\limits_{i<j}^n |s_i-s_j|^{-\mu} ds_1\cdots ds_n,  \label{selberg1}\\
& = \prod_{k=0}^{n-1} \frac{\Gamma(1-(k+1)/\tau)}{\Gamma(1-1/\tau)}
\frac{\Gamma(1+2n/\tau-k/\tau)\Gamma(1-k/\tau)}
{\Gamma(2+n/\tau-(k-1)/\tau)}. \label{mom1plus}
\end{align}
\end{corollary}
\begin{corollary}[Negative integral moments]\label{negmoments}
Let $n\in\mathbb{N}.$
\begin{equation}
\lim\limits_{t\rightarrow\infty}  e^{\mu(\log \varepsilon(t)-\kappa)\frac{n(n-1)}{2}} 
{\bf E} \Bigl[\Bigl(\int_0^1 u^{\mu n} e^{S^{(1)}_t(\mu,u)} du\Bigr)^{-n}\Bigr]  =
\prod_{k=0}^{n-1}
\frac{\Gamma\bigl(2+(n+2+k)/\tau\bigr)
\Gamma\bigl(1-1/\tau\bigr) }{
\Gamma^2\bigl(1+(k+1)/\tau\bigr)
\Gamma\bigl(1+k/\tau\bigr) }.\label{mom2minus}
\end{equation}
Let $n\in\mathbb{N}$ such that $n<(\tau+1)/2.$
\begin{equation}
\lim\limits_{t\rightarrow\infty} e^{\mu(\log \varepsilon(t)-\kappa)\frac{n(n-1)}{2}} 
{\bf E} \Bigl[\Bigl(\int_0^1 e^{S^{(1)}_t(\mu,u)} du\Bigr)^{-n}\Bigr] = 
\prod_{k=0}^{n-1}
\frac{\Gamma\bigl(2+(-n+2+k)/\tau\bigr)
\Gamma\bigl(1-1/\tau\bigr) }{
\Gamma\bigl(1-2n/\tau+(k+1)/\tau\bigr)\Gamma\bigl(1+(k+1)/\tau\bigr)
\Gamma\bigl(1+k/\tau\bigr) }. \label{mom1minus}
\end{equation}
\end{corollary}
\begin{corollary}[Joint integral moments]\label{jointintegralmom}
Let $n, m\in\mathbb{N}$ and denote $N\triangleq n+m,$
$N<2/\tau.$ Let $I_1$ and $I_2$ be non-overlapping subintervals of the unit
interval. Then,
\begin{gather}
\lim\limits_{t\rightarrow\infty} e^{\mu(\log \varepsilon(t)-\kappa)\frac{N(N+1)}{2}}  
{\bf E} \Bigl[\Bigl(\int_{I_1} u^{-\mu N}e^{S^{(1)}_t(\mu,u)} du\Bigr)^n
\Bigl(\int_{I_2} u^{-\mu N}e^{S^{(1)}_t(\mu,u)} du\Bigr)^m\Bigr] = 
\int\limits_{I_1^n \times I_2^m}  \prod\limits_{i<j}^N |s_i-s_j|^{-\mu} ds_1\cdots ds_N, \\
\lim\limits_{t\rightarrow\infty} e^{\mu(\log \varepsilon(t)-\kappa)\frac{N(N+1)}{2}}  
{\bf E} \Bigl[\Bigl(\int_{I_1} e^{S^{(1)}_t(\mu,u)} du\Bigr)^n
\Bigl(\int_{I_2} e^{S^{(1)}_t(\mu,u)} du\Bigr)^m\Bigr] = 
\int\limits_{I_1^n \times I_2^m} \prod_{i=1}^N
s_i^{\mu\,N} \prod\limits_{i<j}^N |s_i-s_j|^{-\mu} ds_1\cdots ds_N .
\end{gather}
\end{corollary}
\begin{corollary}[Asymptotic expansions]\label{asymptotic}
Given $q\in\mathbb{C},$ then in the limit $\mu\rightarrow 0$ we have the
asymptotic expansions\footnote{As usual, $\zeta(s)$ denotes the Riemann zeta function and $B_n(s)$ the $n$th Bernoulli polynomial. Note that $\zeta(1)$ never enters any of the final formulas as the coefficient it multiplies is identically zero.}
\begin{gather}
\lim\limits_{t\rightarrow\infty} e^{\mu(\log \varepsilon(t)-\kappa)\frac{q(q+1)}{2}} 
{\bf E} \Bigl[\Bigl(\int_0^1 u^{-\mu q} e^{S^{(1)}_t(\mu,u)} du\Bigr)^q\Bigr]\thicksim \exp\Bigl(\sum\limits_{r=0}^\infty
\Bigl(\frac{\mu}{2}\Bigr)^{r+1}
\frac{1}{r+1}\Bigl[\zeta(r+1)\times  \nonumber \\
\times\bigl[\frac{B_{r+2}(q+1)
+2B_{r+2}(q) -3B_{r+2}}{r+2} -q\bigr]  
+\bigl(\zeta(r+1)-1\bigr)\bigl[\frac{B_{r+2}(q-1)-B_{r+2}(2q-1)}{r+2}\bigr]\Bigr]\Bigr)
\label{asymptotic2}.
\end{gather}
\begin{gather}
\lim\limits_{t\rightarrow\infty} e^{\mu(\log \varepsilon(t)-\kappa)\frac{q(q+1)}{2}} 
{\bf E} \Bigl[\Bigl(\int_0^1 e^{S^{(1)}_t(\mu,u)} du\Bigr)^q\Bigr] \thicksim \exp\Bigl(\sum\limits_{r=0}^\infty
\Bigl(\frac{\mu}{2}\Bigr)^{r+1}
\frac{1}{r+1}\Bigl[   \zeta(r+1)\bigl[\frac{B_{r+2}(q+1)
+B_{r+2}(q)}{r+2} + \nonumber \\
+ \frac{B_{r+2}(-q)-B_{r+2}(-2q)-2B_{r+2}}{r+2} -q\bigr]  
+\bigl(\zeta(r+1)-1\bigr)\bigl[\frac{B_{r+2}(-q-1)-B_{r+2}(-1)}{r+2}\bigr]\Bigr]\Bigr)
\label{asymptotic1}.
\end{gather}
\end{corollary}
\begin{corollary}[Covariance structure]\label{covstructure}
Let $0<s_1<s_2<1.$ Then, in the limit $\Delta\rightarrow 0,$
\begin{align}
\lim\limits_{t\rightarrow\infty} \Bigl[\mu(\log\varepsilon(t)-\kappa)+{\bf Cov}\Bigl(\log \int\limits_{s_1}^{s_1+\Delta} e^{S^{(1)}_t(\mu,u)} du, \, \log \int\limits_{s_2}^{s_2+\Delta} e^{S^{(1)}_t(\mu,u)} du\Bigr)\Bigr]
= &-\mu\bigr(\log |s_1-s_2|-\log|s_1|-\nonumber \\ &-\log|s_2|\bigl) +O(\Delta),\\
\lim\limits_{t\rightarrow\infty} \Bigl[\mu(3\log\varepsilon(t)-\kappa)+{\bf Cov}\Bigl(\log \int\limits_{s_1}^{s_1+\Delta} e^{S^{(2)}_t(\mu,u)} du, \, \log \int\limits_{s_2}^{s_2+\Delta} e^{S^{(2)}_t(\mu,u)} du\Bigr)\Bigr]
= & -\mu\log |s_1-s_2| +O(\Delta).
\end{align}
\end{corollary}
The next two results deal with the probabilistic structure
that is underlying the Mellin-type transform in \eqref{M2}.
\begin{corollary}[Non-central limit]\label{noncentral}
The limit in \eqref{M2} is the Mellin transform of a positive probability distribution.  Call it $M_\mu.$ Then, for $\Re(q)<\tau,$
\begin{equation}\label{Mlaw}
\lim\limits_{t\rightarrow\infty}  e^{\mu(\log \varepsilon(t)-\kappa)\frac{q(q+1)}{2}} 
{\bf E} \Bigl[\Bigl(\int_0^1 u^{-\mu q} e^{S^{(1)}_t(\mu,u)} du\Bigr)^q\Bigr] = {\bf E}[M_\mu^q].
\end{equation}
$\log M_\mu$ is infinitely divisible on the real line having the L\'evy-Khinchine decomposition
$\log{\bf E}[M_\mu^q]  =
q\,\mathfrak{m}(\mu) + \frac{1}{2} q^2
\sigma^2(\mu)+\int_{\mathbb{R}\setminus \{0\}} \bigl(e^{q u}-1-q
u/(1+u^2)\bigr) d\mathcal{M}_{\mu}(u)$ for
some $\mathfrak{m}(\mu)\in\mathbb{R}$ and the following gaussian
component and spectral function
\begin{align}
\sigma^2(\mu) & = \frac{4\log 2}{\tau}, \\
\mathcal{M}_\mu(u) & = \int\limits_u^\infty 
\Bigl[\frac{\bigl(e^x+2+e^{-x(1+\tau)}\bigr)}
{\bigl(e^x-1\bigr)(e^{x\tau}-1)}
-\frac{e^{-x(1+\tau)/2}}{\bigl(e^{x/2}-1\bigr)(e^{x\tau/2}-1)}\Bigr]
\frac{dx}{x}
\end{align}
for $u>0,$ and $\mathcal{M}_\mu(u)=0$ for $u<0.$ $M_\mu$ is a product of a lognormal, Fr\'echet and independent Barnes beta random variables.\footnote{This decomposition is shown in detail in Section 3, cf. Theorem \ref{B2}.}
\end{corollary}
\begin{corollary}[Multifractality]\label{multifrac}
Let $0<s<1.$ Then, for $\Re(q)<\tau,$ the limit
\begin{equation}
\lim\limits_{t\rightarrow\infty} e^{\mu(\log \varepsilon(t)-\kappa)\frac{q(q+1)}{2}}  
{\bf E} \Bigl[\Bigl(\int_0^s u^{-\mu q} e^{S^{(1)}_t(\mu,u)} du\Bigr)^q\Bigr] = {\bf E}\bigl[M_\mu^q(s)\bigr]
\end{equation}
is the Mellin transform of a probability distribution, call it $M_\mu(s),$
which satisfies the multifractal law\footnote{Also referred to in the literature as stochastic self-similarity or scale-consistent continuous dilation invariance. It
determines the law of $M_\mu(s),$ $s<1,$ from the law of $M_\mu$ in \eqref{Mlaw}
.}
\begin{equation}\label{multifractal}
M_{\mu}(s) \overset{{\rm in \,law}}{=} s\exp\bigl(\Omega_s\bigr) M_{\mu},
\end{equation}
where $\Omega_s$ is a gaussian random variable with the mean $(\mu/2)\log s$ and variance $-\mu\log s$ that is independent of $M_{\mu}$ and $M_\mu$ is as in Corollary \ref{noncentral}. This law is understood as the equality of random variables in law at fixed $s<1.$ In particular, $M_\mu(s)$
is also log-infinitely divisible.
\end{corollary}
\begin{corollary}[Multiscaling]\label{Multi}
Let $0<s<1.$  Then, for $0<\Re(q)<\tau,$
\begin{align}
\lim\limits_{t\rightarrow\infty} e^{\mu(\log \varepsilon(t)-\kappa)\frac{q(q+1)}{2}}  
{\bf E} \Bigl[\Bigl(\int_0^s u^{-\mu q} e^{S^{(1)}_t(\mu,u)} du\Bigr)^q\Bigr]  = const(\mu, q) \,\,s^{q-\frac{\mu}{2}(q^2-q)}, \label{multilaw2}\\
\lim\limits_{t\rightarrow\infty} e^{\mu(\log \varepsilon(t)-\kappa)\frac{q(q+1)}{2}}  
{\bf E} \Bigl[\Bigl(\int_0^s e^{S^{(1)}_t(\mu,u)} du\Bigr)^q\Bigr]  = const(\mu, q)\, \,s^{q+\frac{\mu}{2}(q^2+q)},\label{multilaw1}
\end{align}
\emph{i.e.} these Mellin-type transforms are multiscaling 
as functions of $s.$
\end{corollary}

The rationale for considering two separate transforms in Conjectures \ref{Mellinvar} and \ref{Mellin} is that they are complementary. The transforms in \eqref{M2transf} and \eqref{M2v} are not, strictly speaking, rescaled Mellin transforms because of 
the $u^{-\mu q}$ multiplier. This multiplier is introduced to obtain the Mellin transform of a probability distribution on the right-hand side of these equations. If we drop this multiplier, we loose this probabilistic interpretation but gain a \emph{bona fide} rescaled Mellin transforms of $\int_0^1 e^{S^{(1)}_t(\mu,u,\varepsilon)} du$ and $\int_0^1 e^{S^{(1)}_t(\mu,u)} du$ on the left-hand side 
so that the limit fits into the general theory of mod-gaussian convergence, cf. \cite{Jacod}. The interest in introducing
the second statistic $S^{(2)}_t(\mu,u,\varepsilon)$ and $S^{(2)}_t(\mu,u)$ is precisely to eliminate the need to introduce the $u^{-\mu q}$ multiplier so
that the transforms in \eqref{M2transfstrong} and \eqref{M2vstrong} are both \emph{bona fide} rescaled Mellin transforms
and give the Mellin transform of a probability distribution on the right-hand side of these equations. 
We note that the right-hand sides of both \eqref{M2} and \eqref{M1} are special cases of the Mellin transform of
the Selberg integral distribution that is given in \eqref{Mfunc} and \eqref{ShintaniFactor}. The meaning of the conditions $\Re(q)<\tau$ and
$-(\tau+1)/2<\Re(q)<\tau$ 
is that
the right-hand sides of \eqref{M2} and \eqref{M1} are analytic and zero-free over these regions. As will be explained in Section 4, one can insert the $u^{\lambda_1}(1-u)^{\lambda_2}$ prefactor into the exponential functionals on the left-hand sides of \eqref{M2transf} -- \eqref{M1vstrong} so as to obtain the general Mellin transform of the Selberg integral distribution on the right-hand side, cf. Theorem \ref{general} for details. We chose not to state the most general case here to avoid unnecessary complexity.

\begin{remark}
A random matrix analogue of the Bourgade-Kuan-Rodgers theorems was established in \cite{FKS}. The corresponding linear statistic
is an appropriately centered log-absolute value of the characteristic polynomial of a suitably scaled GUE matrix. The eigenvalues
of the matrix replace the roots and its dimension $N$ replaces $\log t.$ It is shown in \cite{FKS} that the $N\rightarrow\infty$ 
limit of the GUE statistic is gaussian, and the limiting covariance can be shown\footnote{Nickolas Simm, private communication.} to be equivalent to the Bourgade-Kuan-Rodgers covariance. 
As it will become clear in Section 4, our results apply to any linear statistic that is asymptotically gaussian with the Bourgade-Kuan-Rodgers covariance. 
Thus, under the correspondence $N\thicksim\log t,$ they apply equally to the GUE statistic. The significance of this observation is that
the convergence of our statistics to the gaussian free field or its centered version provides a theoretical explanation for why 
our conjecture about the zeroes can be approached directly from the GUE side. 
\end{remark}
\begin{remark}
Our conjecture has only GUE terms and no arithmetic corrections, contrary to most of the conjectures that we mentioned in the Introduction.  
The reason for this is the choice of the mesoscopic scale $1\ll \lambda(t).$ As Fujii \cite{Fujii2} shows in his analysis of Berry's conjecture, 
which in particular deals with the second moment of the fluctuation of the error term, cf. \eqref{StS}, the GUE term in Berry's formula 
dominates arithmetic corrections precisely under the condition $1\ll \lambda(t)$. As our conjecture involves essentially the same statistic, 
except for smoothing, we expect that in our case this condition achieves the same effect of dominating arithmetic terms.
\end{remark} 
\begin{remark}
The interest in the strong version of our conjecture is that it contains information about the statistical distribution of the zeroes at large but finite $t,$
whereas the weak version only describes the distribution at $t=\infty.$ Indeed, as it will be explained in Section 4, the strong version is equivalent to the weak version provided 
the statistics $S^{(i)}_t(\mu,u)$ converge to their gaussian limits faster than $1/|\log\varepsilon(t)|,$ 
in the sense of the rate of convergence of the variance to its asymptotic value, and this is expected to be the case due to the asymptotic condition
in \eqref{vart}. 
Moreover,
as the strong conjecture fits into the general framework of mod-gaussian convergence, the results of \cite{Feray} and \cite{Meliot} and the explicit knowledge of our
limiting functions make it possible to quantify the normality zone, \emph{i.e.} the scale up to which the tails of our exponential functionals are normal,
and the breaking of symmetry near the edges of the normality zone thereby quantifying precise deviations at large $t.$ The actual computation
of these quantities is left for future research. 
\end{remark}

\section{A Review of the Limit Lognormal Measure and Selberg Integral Distribution}
\noindent In this section we will give a self-contained review of the limit lognormal measure of Mandelbrot-Bacry-Muzy on the unit interval and of the Selberg integral probability distribution mainly following our earlier presentations in \cite{Me4}, \cite{MeIMRN}, and \cite{Me14}.
We will indicate what is known and what is conjectured and refer the reader
to appropriate original publications for the proofs. 
 
The limit lognormal measure (also known as lognormal multiplicative chaos) 
is defined as the exponential functional of the gaussian free field.
Let $\omega_{\mu, L, \varepsilon}(s)$ be a stationary gaussian process in $s,$ whose mean and covariance are functions of three parameters $\mu>0,$ $L>0,$ and $\varepsilon>0.$ We
consider the random measure 
\begin{equation}\label{Mmuvarepsilon}
M_{\mu, L, \varepsilon}(a, \,b)=\int\limits_a^b e^{\omega_{\mu, L, \varepsilon}(s)} \, ds.
\end{equation}
The mean and covariance of $\omega_{\mu, L, \varepsilon}(s)$ are defined to be,
cf. \cite{BM},
\begin{align}
{\bf{E}} \left[\omega_{\mu, L, \varepsilon}(t)\right] & =  -\frac{\mu}{2} \,
\left(1+\log\frac{L}{\varepsilon}\right), \label{mean} \\
{\bf{Cov}}\left[\omega_{\mu, L, \varepsilon}(t), \,\omega_{\mu, L, \varepsilon}
(s)\right] &  = 
\mu \, \log\frac{L}{|t-s|}, \, \varepsilon\leq|t-s|\leq L, \label{cov} \\
{\bf{Cov}}\left[\omega_{\mu, L, \varepsilon}(t),
\,\omega_{\mu, L, \varepsilon}(s)\right] & =  \mu
\left(1+\log\frac{L}{\varepsilon}-\frac{|t-s|}{\varepsilon}\right),
\, |t-s|<\varepsilon, \label{cov2}
\end{align}
and covariance is zero in the remaining case of $|t-s|\geq L.$
Thus, $\varepsilon$ is used as a truncation scale. $L$ is the decorrelation length of the process. $\mu$ is the intermittency parameter (also known as inverse temperature in the physics literature). The two key properties of this construction are, first, that
\begin{equation}\label{normalization}
{\bf{E}}
\bigl[\omega_{\mu, L, \varepsilon}(t)\bigr]=-\frac{1}{2}{\bf{Var}}\bigl[\omega_{\mu, L,\varepsilon}(t)\bigr]
\end{equation}
so that
\begin{equation}\label{normalization2}
{\bf E}\bigl[\exp\bigl(\omega_{\mu, L, \varepsilon}
(s)\bigr)\bigr]=1
\end{equation}
and, second, that
\begin{equation}
{\bf{Var}}\bigl[\omega_{\mu, L, \varepsilon}(t)\bigr] \propto -\log\varepsilon
\end{equation}
is logarithmically divergent as $\varepsilon\rightarrow 0.$ The first property is essential for convergence, the second is responsible for
multifractality, and both are originally due to Mandelbrot
\cite{secondface}. 

The interest in the limit lognormal
construction stems from the $\varepsilon\rightarrow 0$ limit. It is clear that
the $\varepsilon\rightarrow 0$ limit of $\omega_{\mu, L, \varepsilon}(t)$ does not exist as a stochastic process (this limiting ``process'' is known as the gaussian free field). Remarkably, using the theory of $T$-martingales developed by Kahane \cite{K2}, the work of Barral
and Mandelbrot \cite{Pulses} on log-Poisson cascades, and conical constructions of Rajput and Rosinski \cite{RajRos} and Marsan and Schmitt \cite{Schmitt},
Bacry and
Muzy\,\cite{BM1} showed that $M_{\mu, L, \varepsilon}(dt)$ converges weakly (as
a measure on $\mathbb{R}^+$) a.s. to a non-trivial random limit measure $M_{\mu, L}(dt)$ 
\begin{equation}
M_{\mu, L}(a, b)=\lim\limits_{\varepsilon\rightarrow 0} M_{\mu, L, \varepsilon}(a, b),
\end{equation}
provided $0\leq \mu<2,$ and the limit is stationary $M_{\mu, L}(t,t+\tau) \overset{{\rm in \,law}}{=} M_{\mu, L}(0,\tau)$ and
non-degenerate in the sense that ${\bf{E}}[M_{\mu, L}(0, \,t)]=t.$ We will denote the total random mass of $[0, L]$ by $M_{\mu, L},$ 
\begin{equation}
M_{\mu, L} \triangleq M_{\mu, L}(0, \,L).
\end{equation}
It is shown in \cite{BM1} that for $q>0$ we have 
\begin{equation}
q<\frac{2}{\mu} \Longrightarrow {\bf{E}}\bigl[M^q_{\mu, L}\bigr]<\infty \;\text{and}
\; q>\frac{2}{\mu} \Longrightarrow {\bf{E}}\bigl[M^q_{\mu, L}\bigr]=\infty
.
\end{equation}

The fundamental property of the limit measure is that it is multifractal. This can be understood at several levels, for our purposes, this means that its total
mass exhibits stochastic self-similarity, also known as
continuous dilation invariance, as first established in \cite{MRW}. Given $\gamma<1,$ let $\Omega_\gamma$ denote a
gaussian random variable that is independent of the process $\omega_{\mu, L, \varepsilon}(s)$ 
such that
\begin{align}
{\bf E}\bigl[\Omega_\gamma\bigr] & = \frac{1}{2}\mu\log \gamma, \label{OmegaE} \\ 
{\bf Var}\bigl[\Omega_\gamma\bigr] & = -\mu\log \gamma. \label{OmegaV}
\end{align}
Then, there hold the following invariances\footnote{The first invariance was  discovered in \cite{MRW} and later generalized in \cite{BM}. We discovered the other two invariances in \cite{Me2} and developed a general theory of such invariances in \cite{Me5}.}
\begin{align}
\Omega_\gamma + \omega_{\mu, \,L, \,\varepsilon}(s)
& \overset{{\rm in \,law}}{=} 
\omega_{\mu, \, L,\, \gamma\varepsilon}(\gamma s), \label{invariance} \\
 \Omega_\gamma + \omega_{\mu, \,L,\, \varepsilon}(s)
 & \overset{{\rm in \,law}}{=} 
\omega_{\mu,\, L/\gamma,\, \varepsilon}(s), \label{invariance2} \\
\Omega_\gamma + \omega_{\mu,\, L,\, \varepsilon}(s)
& \overset{{\rm in \,law}}{=}
\omega^{(1)}_{\mu(1+\log\gamma), \,L,\, \varepsilon}(s) +
\omega^{(2)}_{-\mu\log\gamma,\, eL,\, \varepsilon}(s), \label{invariance3}
\end{align}
that are understood as equalities in law of stochastic processes in $s$ on the interval $[0, L]$ at fixed $\varepsilon,$ $L,$ and $0<\gamma<1.$ In \eqref{invariance3} the superscripts denote independent copies of the free field, $e$ stands for the base of the natural logarithm, and $e^{-1}<\gamma<1.$ The truncation scale invariance in \eqref{invariance}
implies the multifractal law of the limit measure\footnote{It must be emphasized that this equality is strictly in law, that is, $\Omega_\gamma$ is not a stochastic process, \emph{i.e.} $\Omega_\gamma$ and $\Omega_{\gamma '}$ for $\gamma\neq \gamma '$ are not defined on the same probability space. In particular, \eqref{multifractallaw} determines the distribution of $M_{\mu, L}(0, \,t)$ in terms of the law of the total mass but says nothing about the latter or their joint distribution.}
for $t<L$
\begin{equation}\label{multifractallaw}
M_{\mu, L}(0,\,t) \overset{{\rm in \,law}}{=} \frac{t}{L}\,e^{\Omega_{\frac{t}{L}}}\,M_{\mu, L},
\end{equation}
which implies that the moments of the total mass obey for $0<q<2/\mu$ the multiscaling law
\begin{equation}\label{multi}
{\bf E}\bigl[M_{\mu, L}(0,\,t)^q\bigr] \propto \Bigl(\frac{t}{L}\Bigr)^{q-\frac{\mu}{2}(q^2-q)} 
\end{equation}
as a function of $t<L.$ The decorrelation length invariance in \eqref{invariance2} implies that the dependence of the total mass on $L$ is trivial,\footnote{This invariance also determines how the law of the total mass  behaves under a particular change of probability measure, cf. \cite{Me5} for details.} 
\begin{equation}
M_{\mu, L} \overset{{\rm in \,law}}{=} L\,M_{\mu, 1},
\end{equation}
so that we can restrict ourselves to $L=1$ without a loss of generality.
Henceforth, 
\begin{equation}
L=1,
\end{equation}
and $L$ is dropped from all subsequent formulas. Finally, 
the significance of the intermittency invariance in \eqref{invariance3} 
is that it gives the rule of intermittency differentiation and effectively determines the law of the total mass, cf. Theorem \ref{Differentiation} and the  discussion following it below.

The law of the total mass can be reformulated as a
non-central limit problem. Let us break up the unit interval into the subintervals
of length $\varepsilon$ so that $s_j=j\varepsilon,$
$\omega_j\triangleq \omega_{\mu, \varepsilon} (s_j),$ and $N\varepsilon=1.$
It is shown in \cite{BM1} that the total mass $M_\mu$ can be
approximated as
\begin{equation}\label{Msum}
\varepsilon\sum\limits_{j=0}^{N-1} e^{\omega_j} \rightarrow
M_\mu\,\,{\text as}\,\, \varepsilon\rightarrow 0.
\end{equation}
The essence of this result is that the limit is not affected by one's truncation
of covariance so long as \eqref{normalization} holds. This representation is quite useful in calculations. In particular, it is easy to see that it implies an important relationship between the moments of the generalized total mass\footnote{By a slight abuse of terminology, we refer to any integral of the form $\int_0^1 \varphi(s)\,M_\mu(ds)$ as the generalized total mass.}  of the limit measure and a class of (generalized) Selberg integrals, originally due to \cite{MRW}. Let $\varphi(s)$ be an appropriately chosen test function and $I$ a subinterval of the unit interval, then
\begin{equation}
{\bf{E}} \Bigl[\bigl(\int\limits_I \varphi(s) \,
M_\mu(ds)\bigr)^n\Bigr] = \int\limits_{I^n} \prod_{i=1}^n
\varphi(s_i) \prod\limits_{i<j}^n
|s_i-s_j|^{-\mu} ds_1\cdots ds_n.
\end{equation}
More generally, the same type of result holds for any finite number of
subintervals of the unit interval. For example, the joint $(n,m)$ moment is given by a generalized Selberg integral of dimension $n+m$
\begin{align}
{\bf E}\Bigl[\Bigl(\int\limits_{I_1} \varphi_1(s)\,M_\mu(ds)\Bigr)^{n}\Bigl(\int\limits_{I_2} \varphi_2(s)\,M_\mu(ds)\Bigr)^{m}\Bigr] =  
\int\limits_{I_1^n\times I_2^m} & \prod\limits_{i=1}^{n} \varphi_1(s_i) \prod\limits_{i=n+1}^{n+m} \varphi_2(s_i) \prod\limits_{k<p}^{n+m}|t_p-t_k|^{-\mu}
ds_1\cdots ds_{n+m}.\label{jointmomentlimitmeasure}
\end{align}
It can be shown as a corollary of this equation, cf. \cite{MRW}, that the covariance structure of the total mass is logarithmic. Given $0<s<1,$ 
\begin{equation}\label{covstrucmass}
{\bf Cov}\Bigl(\log\int\limits_s^{s+\Delta s}M_\mu(du), \,\log\int\limits_0^{\Delta s} M_\mu(du)\Bigr) = -\mu\log s + O(\Delta s),
\end{equation}
which indicates that the mass of non-overlapping subintervals of the unit interval exhibits strong stochastic dependence. 
In the special case of $\varphi(s)=s^{\lambda_1}(1-s)^{\lambda_2},$ we have an explicit formula for moments of order $n<2/\mu,$ as was first pointed out in \cite{BDM}, that is given by the classical Selberg integral, cf. Chapter 4 of \cite{ForresterBook} for a modern treatment and \cite{Selberg} for the original derivation,
\begin{equation}\label{selbergformula}
{\bf{E}} \Bigl[\bigl(\int_0^1 s^{\lambda_1}(1-s)^{\lambda_2} \,
M_\mu(ds)\bigr)^n\Bigr] 
=
\prod_{k=0}^{n-1}
\frac{\Gamma(1-(k+1)\mu/2)\Gamma(1+\lambda_1-k\mu/2)\Gamma(1+\lambda_2-k\mu/2)}
{\Gamma(1-\mu/2)\Gamma(2+\lambda_1+\lambda_2-(n+k-1)\mu/2)}.
\end{equation} 
The law of $\int_0^1 s^{\lambda_1}(1-s)^{\lambda_2} \,
M_\mu(ds)$ is not known rigorously, even for $\lambda_1=\lambda_2=0,$ \emph{i.e.} the total mass. It is possible to derive it heuristically so as to
formulate a precise conjecture about it as follows. Consider the expectation of a general functional of the limit lognormal measure
\begin{equation}\label{vfunctional}
v(\mu,\,\varphi,\, f,\,F) \triangleq {\bf E}\Bigl[F\Bigl(\int_0^1 e^{\mu
f(s)}\varphi(s) \,M_\mu(ds)\Bigr)\Bigr].
\end{equation}
The integration with
respect to $M_\mu(ds)$ is understood in the sense of
$\varepsilon\rightarrow 0$ limit so that
$v(\mu,\,\varphi,\,f,\,F)=\lim\limits_{\varepsilon\rightarrow
0}v_\varepsilon(\mu,\,\varphi,\,f,\,F)$ and $v_\varepsilon(\mu,\,\varphi,\, f,\,F)
\triangleq {\bf E}\Bigl[F\Bigl(\int_0^1 e^{\mu f(s)} \varphi(s)\,
M_{\mu, \varepsilon}(ds)\Bigr)\Bigr]$ with $M_{\mu,\varepsilon}(ds)$ as in
\eqref{Mmuvarepsilon}. Also, let $g(s_1,\,s_2)$ be defined by
\begin{equation}
g(s_1, \,s_2) \triangleq -\log|s_1-s_2|.
\end{equation}
Finally, we will use $[0, \,1]^k$ to denote the $k-$dimensional unit
interval $[0,\,1]\times\cdots\times [0,\,1].$ Then, we have the
following rule of intermittency differentiation, cf. \cite{Me3} and \cite{MeIMRN} for derivations and \cite{Me6} for an extension to
the joint distribution of the mass of multiple subintervals of the unit interval.
\begin{theorem}[Intermittency differentiation]\label{Differentiation}
The expectation $v(\mu,\,\varphi,\, f,\,F)$ is invariant under
intermittency differentiation and satisfies
\begin{gather} \frac{\partial}{\partial \mu} v(\mu,\,\varphi,\,
f,\,F) = \int\limits_{[0,\,1]} v\bigl(\mu,\,\varphi,\, f+g(\cdot,
s),\,F^{(1)}\bigr) e^{\mu f(s)}f(s)\varphi(s)\,ds+ \nonumber \\
+\frac{1}{2}\int\limits_{[0,\,1]^2} v\bigl(\mu,\varphi,
f+g(\cdot,s_1)+g(\cdot,s_2),F^{(2)}\bigr)
e^{\mu\bigl(f(s_1)+f(s_2)+g(s_1,s_2)\bigr)}g(s_1,
s_2)\varphi(s_1)\varphi(s_2)\,ds_1\,ds_2. \label{IntermDiff}
\end{gather}
\end{theorem}
The mathematical content of \eqref{IntermDiff} is that differentiation with
respect to the intermittency parameter $\mu$ is equivalent to a
combination of two functional shifts induced by the $g$ function. This differentiation rule is nonlocal as it involves the entire path of the process $s\rightarrow M_\mu(0,\,s),$ $s\in(0,\,1).$ It
is clear that both terms in \eqref{IntermDiff} are of the same functional form
as the original functional in \eqref{vfunctional} so that Theorem \ref{Differentiation} allows us
to compute derivatives of all orders. There results the following
formal expansion with some coefficients $H_{n,k}(\varphi)$ that are
independent of $F.$ Let
\begin{gather}
S_l(\mu,\varphi) \triangleq \int\limits_{[0,\,1]^l} \prod_{i=1}^l
\varphi(s_i)\, \prod\limits_{i<j}^l |s_i-s_j|^{-\mu} ds_1\cdots
ds_l, \\
x\triangleq \int\limits_0^1 \varphi(s) \,ds.
\end{gather}
Then, we obtain the formal intermittency expansion
\begin{equation}\label{Expansion}
{\bf E}\Bigl[F\Bigl(\int_0^1 \varphi(s)
\,M_\mu(ds)\Bigr)\Bigr]=F(x)+ \sum\limits_{n=1}^\infty
\frac{\mu^n}{n!}\Bigl[ \sum_{k=2}^{2n} F^{(k)}(x)
H_{n,k}(\varphi)\Bigr].
\end{equation}
The expansion coefficients $H_{n, k}(\varphi)$
are given by the binomial transform of the derivatives of the
positive integral moments
\begin{equation}
H_{n, k}(\varphi) = \frac{(-1)^k}{k!} \sum\limits_{l=2}^k (-1)^l
\binom{k}{l}\,x^{k-l}\,\frac{\partial^n S_l}{\partial
\mu^n}\vert_{\mu=0},
\end{equation}
and satisfy the identity
\begin{equation}
H_{n, k}(\varphi) = 0\,\,\,\forall k>2n.
\end{equation}
\begin{remark}
The last equation and Theorem \ref{Differentiation} say that the intermittency expansion in \eqref{Expansion} is an exactly renormalized expansion in the centered moments of
$\int_0^1 \varphi(s) \,dM_\mu(s).$ Indeed, we have the identity
\begin{equation}
\partial^n/\partial \mu^n \,\vert_{\mu=0} \,{\bf
E}\Bigl[\bigl(\int_0^1 \varphi(s) \,M_\mu(ds)-x\bigl)^k\Bigr]=
k!\,H_{n,k}(\varphi).
\end{equation}
\end{remark}

From now on we will focus on
\begin{equation}
\varphi(s) \triangleq s^{\lambda_1}(1-s)^{\lambda_2},
\;\lambda_1,\,\lambda_2>-\frac{\mu}{2},
\end{equation}
which corresponds to the full Selberg integral. Clearly, we have
\begin{equation}
x =
\frac{\Gamma(1+\lambda_1)\Gamma(1+\lambda_2)}{\Gamma(2+\lambda_1+\lambda_2)}.
\end{equation}
The moments $S_l(\mu, \,\varphi)=S_l(\mu, \,\lambda_1, \,\lambda_2)$ are given
by Selberg's product formula in \eqref{selbergformula}. By expanding $\log S_l(\mu,
\,\lambda_1, \,\lambda_2)$ in powers of $\mu$ near zero and 
computing the resulting $H_{n,k}(\varphi)$ coefficients,
we derived in \cite{Me4} and \cite{MeIMRN} the following expansion for the Mellin transform
in terms of Bernoulli polynomials and values of the Hurwitz zeta function at the integers
\begin{align}
{\bf E}\Bigl[\Bigl(\int_0^1 s^{\lambda_1} (1-s)^{\lambda_2} &
\,M_\mu(ds)\Bigr)^q\Bigr] = x^q \exp\Bigl(\sum_{r=0}^\infty
\frac{\mu^{r+1}}{r+1} \,b_r(q, \,\lambda_1,\,\lambda_2)\Bigr), \label{fullasymptseries}\\
b_r(q, \,\lambda_1,\,\lambda_2) & \triangleq \frac{1}{2^{r+1}}
\Bigl[\bigl(\zeta(r+1, 1+\lambda_1)+\zeta(r+1,
1+\lambda_2)\bigr)\Bigl(\frac{B_{r+2}(q)-B_{r+2}}{r+2}\Bigr) -
\nonumber \\ & -\zeta(r+1)q + \zeta(r+1)
\Bigl(\frac{B_{r+2}(q+1)-B_{r+2}}{r+2}\Bigr) - \nonumber \\ & -
\zeta(r+1, 2+\lambda_1+\lambda_2)
\Bigl(\frac{B_{r+2}(2q-1)-B_{r+2}(q-1)}{r+2}\Bigr)\Bigr].
\end{align}
The series in \eqref{fullasymptseries} is generally divergent\footnote{It is convergent for ranges of integral $q$ as shown in \cite{Me4}.} and interpreted as the
asymptotic expansion of the Mellin transform in the limit $\mu\rightarrow 0.$
The intermittency expansion of the Mellin transform implies 
a similar expansion of the general transform of
$\log\widetilde{M}_\mu(\lambda_1,\,\lambda_2),$ 
where we introduced the normalized distribution
\begin{equation}
\widetilde{M}_\mu (\lambda_1,\,\lambda_2)\triangleq
\frac{\Gamma(2+\lambda_1+\lambda_2)}{\Gamma(1+\lambda_1)\Gamma(1+\lambda_2)}
\int_0^1 s^{\lambda_1} (1-s)^{\lambda_2} \,M_\mu(ds).
\end{equation}
Then, given constants $a$ and $s$ and a smooth function
$F(s),$ the intermittency expansion is
\begin{align}
{\bf
E}\Bigl[F\bigl(s+a\log\widetilde{M}_\mu(\lambda_1,\,\lambda_2)\bigr)\Bigr] &
= \sum\limits_{n=0}^\infty
F_n\bigl(a,\,s,\,\lambda_1,\,\lambda_2\bigr) \frac{\mu^n}{n!}, \\
F_0\bigl(a,\,s,\,\lambda_1,\,\lambda_2\bigr) & =F(s), \\
F_{n+1}\bigl(a,\,s,\,\lambda_1,\,\lambda_2\bigr) & =
\sum\limits_{r=0}^n \frac{n!}{(n-r)!}  \, b_r\Bigl(a\frac{d}{ds},
\,\lambda_1,\,\lambda_2\Bigr)
F_{n-r}\bigl(a,\,s,\,\lambda_1,\,\lambda_2\bigr).
\end{align}
The expansion is thus obtained by replacing $q$ with $a\,d/ds$ in the solution
for the Mellin transform in \eqref{fullasymptseries}.
We note parenthetically that the general transform is
particularly interesting in the special case of a purely imaginary
$a,$ in which case the operator $F(s)\rightarrow {\bf
E}\Bigl[F\bigl(s+a\log\widetilde{M}_\mu(\lambda_1,\,\lambda_2)\bigr)\Bigr]$
is formally self-adjoint.

The calculation of the intermittency expansion of the Mellin transform 
naturally poses the problem of constructing a positive probability distribution
such that its positive integral moments are given by Selberg's formula, cf.
\eqref{selbergformula}, and the asymptotic expansion of its Mellin transform 
coincides with the intermittency expansion in \eqref{fullasymptseries}. Equivalently, one wants to construct a meromorphic function $\mathfrak{M}(q\,|\,\mu,\lambda_1,\lambda_2)$
that (1) recovers Selberg's formula for positive integral $q<2/\mu,$ (2) is the Mellin transform of a probability distribution for 
$\Re(q)<2/\mu$ as long as $0<\mu<2,$ and (3) has the asymptotic
expansion in the limit $\mu\rightarrow 0$ that is given in \eqref{fullasymptseries}.
Such a function can be naturally thought of as an
analytic continuation of the Selberg integral as a function of its dimension to the complex plane. 

We will describe an analytic and a probabilistic solution to this problem that we first found in \cite{Me4} in the special case of $\lambda_1=\lambda_2=0$ and
then in \cite{MeIMRN} in general.\footnote{The general case was first considered by Fyodorov {\it et. al.} \cite{FLDR}, who gave an equivalent expression for the right-hand side of
\eqref{Mfunc} and so recovered the positive integral moments without proving analytically that their formula corresponds to the Mellin transform of a probability distribution or matching the asymptotic expansion, \emph{i.e.} solved (1) only. Instead, they used Selberg's formula to deduce the functional equation in \eqref{funceqM1} for positive integral $q,$ conjectured that it holds for complex $q,$ and then found a meromorphic function that satisfies \eqref{funceqM1}.} Define 
\begin{align}
\mathfrak{M}(q\,|\,\mu,\lambda_1,\lambda_2)  & \triangleq
\tau^{\frac{q}{\tau}} (2\pi)^{q}\,\Gamma^{-q}\bigl(1-1/\tau\bigr)
\frac{\Gamma_2(1-q+\tau(1+\lambda_1)\,|\,\tau)}{\Gamma_2(1+\tau(1+\lambda_1)\,|\,\tau)}
\frac{\Gamma_2(1-q+\tau(1+\lambda_2)\,|\,\tau)}{\Gamma_2(1+\tau(1+\lambda_2)\,|\,\tau)}
\star
\nonumber \\ & \star
\frac{\Gamma_2(-q+\tau\,|\,\tau)}{\Gamma_2(\tau\,|\,\tau)}
\frac{\Gamma_2(2-q+\tau(2+\lambda_1+\lambda_2)\,|\,\tau)}{\Gamma_2(2-2q+\tau(2+\lambda_1+\lambda_2)\,|\,\tau)}
\label{Mfunc}
\end{align}
\begin{theorem}[Existence of the Selberg integral distribution]\label{BSM}
Let $0<\mu<2,$ $\lambda_i>-\mu/2,$ and $\tau=2/\mu.$ Then,
$\mathfrak{M}(q\,|\,\mu,\lambda_1,\lambda_2)$ is the
Mellin transform of a probability distribution on $(0,\infty)$
for $\Re(q)<\tau,$ and its moments satisfy  
\begin{align}
\mathfrak{M}(n\,|\,\mu,\lambda_1,\lambda_2) & = \prod_{k=0}^{n-1}
\frac{\Gamma(1-(k+1)/\tau)\Gamma(1+\lambda_1-k/\tau)\Gamma(1+\lambda_2-k/\tau)}
{\Gamma(1-1/\tau)\Gamma\bigl(2+\lambda_1+\lambda_2-(n+k-1)/\tau\bigr)},
\; 1\leq n<\Re(\tau), \label{posmomentformula} \\
\mathfrak{M}(-n\,|\,\mu,\lambda_1,\lambda_2) & = \prod_{k=0}^{n-1}
\frac{\Gamma\bigl(2+\lambda_1+\lambda_2+(n+2+k)/\tau\bigr)
\Gamma\bigl(1-1/\tau\bigr) }{
\Gamma\bigl(1+\lambda_1+(k+1)/\tau\bigr)\Gamma\bigl(1+\lambda_2+(k+1)/\tau\bigr)
\Gamma\bigl(1+k/\tau\bigr) },\; n\in\mathbb{N}. \label{negmomentformula}
\end{align}
The function $\mathfrak{M}(q\,|\,\mu,\lambda_1,\lambda_2)$ satisfies
the functional equations
\begin{align}
\mathfrak{M}(q\,|\,\mu,\lambda_1,\lambda_2) & = 
\mathfrak{M}(q-1\,|\,\mu,\lambda_1,\lambda_2) \,
\frac{\Gamma(1-q/\tau)\Gamma\bigl(2+\lambda_1+\lambda_2-(q-2)/\tau\bigr)}{\Gamma(1-1/\tau)}
\star \nonumber \\ & \star
\frac{\Gamma\bigl(1+\lambda_1-(q-1)/\tau\bigr)\Gamma\bigl(1+\lambda_2-(q-1)/\tau\bigr)}
{\Gamma\bigl(2+\lambda_1+\lambda_2-(2q-2)/\tau\bigr)\Gamma\bigl(2+\lambda_1+\lambda_2-(2q-3)/\tau\bigr)}. \label{funceqM1}
\end{align}
\begin{align}
\mathfrak{M}(q\,|\,\mu,\lambda_1,\lambda_2) & = 
\mathfrak{M}(q-\tau\,|\,\mu,\lambda_1,\lambda_2)\,\tau
(2\pi)^{\tau-1}\Gamma^{-\tau}\Bigl(1-\frac{1}{\tau}\Bigr)
\Gamma(\tau-q) \star \nonumber \\ &\star
\frac{\Gamma\bigl((1+\lambda_1)\tau-(q-1)\bigr)\Gamma\bigl((1+\lambda_2)\tau-(q-1)\bigr)}{\Gamma\bigl((2+\lambda_1+\lambda_2)\tau-(2q-2)\bigr)}
\frac{\Gamma\bigl((2+\lambda_1+\lambda_2)\tau-(q-2)\bigr)}{\Gamma\bigl((3+\lambda_1+\lambda_2)\tau-(2q-2)\bigr)}, \label{funceqMtau}
\end{align}
The function $\log\mathfrak{M}(q\,|\,\mu,\lambda_1,\lambda_2)$ has
the asymptotic expansion as $\mu\rightarrow +0$
\begin{gather}
\log\mathfrak{M}(q\,|\,\mu,\lambda_1,\lambda_2) \thicksim
q\log\Bigl(\frac{\Gamma(1+\lambda_1)\Gamma(1+\lambda_2)}{\Gamma(2+\lambda_1+\lambda_2)}\Bigr)+
\sum\limits_{r=0}^\infty \Bigl(\frac{\mu}{2}\Bigr)^{r+1}
\frac{1}{r+1}\Bigl[-\zeta(r+1)q+\nonumber
\\+\bigl(\zeta(r+1, 1+\lambda_1)+\zeta(r+1,
1+\lambda_2)\bigr)\Bigl(\frac{B_{r+2}(q)-B_{r+2}}{r+2}\Bigr)
 + \zeta(r+1) \star \nonumber \\
\star\Bigl(\frac{B_{r+2}(q+1)-B_{r+2}}{r+2}\Bigr) - \zeta(r+1,
2+\lambda_1+\lambda_2)
\Bigl(\frac{B_{r+2}(2q-1)-B_{r+2}(q-1)}{r+2}\Bigr)\Bigr].\label{asymptoticgeneral}
\end{gather}
\end{theorem}
The structure of the corresponding probability distribution, which we denote by
$M_{(\mu, \lambda_1, \lambda_2)},$ is most naturally explained using the theory of Barnes beta distributions that we developed in \cite{MeIMRN}, \cite{Me13}. For simplicity, we restrict ourselves here to the special case of the distribution of type $(2,2)$
and refer the reader to \cite{Me13}, \cite{Me14} for the general case. 
\begin{theorem}[Barnes beta distribution of type $(2,2)$]\label{eta22}
Let $b_0, b_1, b_2>0$ and $\tau>0.$ Define
\begin{align}
\eta(q\,|\,\tau, b) & \triangleq
\frac{\Gamma_2(q+b_0\,|\,\tau)}{\Gamma_2(b_0\,|\,\tau)}
\frac{\Gamma_2(b_0+b_1\,|\,\tau)}{\Gamma_2(q+b_0+b_1\,|\,\tau)}
\frac{\Gamma_2(b_0+b_2\,|\,\tau)}{\Gamma_2(q+b_0+b_2\,|\,\tau)}
\frac{\Gamma_2(q+b_0+b_1+b_2\,|\,\tau)}{\Gamma_2(b_0+b_1+b_2\,|\,\tau)}, \\
& = \prod\limits_{n_1, n_2=0}^\infty \Bigl[
\frac{b_0+n_1+n_2\tau}{q+b_0+n_1+n_2\tau}\frac{q+b_0+b_1+n_1+n_2\tau}
{b_0+b_1+n_1+n_2\tau} \frac{q+b_0+b_2+n_1+n_2\tau}
{b_0+b_2+n_1+n_2\tau}\star \nonumber \\
& \star \frac{b_0+b_1+b_2+n_1+n_2\tau}
{q+b_0+b_1+b_2+n_1+n_2\tau}\Bigr]. \label{barnesfactor}
\end{align}
Then, $\eta(q\,|\,\tau ,b)$ is the Mellin transform of a probability
distribution on $(0, 1).$ Denote it by $\beta(\tau, b).$
\begin{equation}
{\bf E}\bigl[\beta(\tau,b)^q\bigr] = \eta(q\,|\,\tau,\,b),\;
\Re(q)>-b_0.
\end{equation}
The distribution $-\log\beta(\tau, b)$ is absolutely continuous and infinitely divisible on
$(0, \infty)$ and has the L\'evy-Khinchine decomposition
\begin{equation}\label{LKH}
{\bf E}\Bigl[\exp\bigl(q\log\beta(\tau, b)\bigr)\Bigr] =
\exp\Bigl(\int\limits_0^\infty (e^{-xq}-1) e^{-b_0
x} \frac{(1-e^{-b_1 x})(1-e^{-b_2 x})}{(1-e^{-x})(1-e^{-\tau x})} \frac{dx}{x}\Bigr),
\; \Re(q)>-b_0.
\end{equation}
\end{theorem}
We can now describe the probabilistic structure of the Selberg integral distribution. Let $\tau>1$ and define a lognormal random variable
\begin{equation}
L \triangleq \exp\bigl(\mathcal{N}(0,\,4\log 2/\tau)\bigr),
\end{equation}
where $\mathcal{N}(0,\,4\log 2/\tau)$ denotes a zero-mean gaussian with variance $4\log 2/\tau,$ and a  Fr\'echet variable $Y$ having density
$\tau\,y^{-1-\tau}\exp\bigl(-y^{-\tau}\bigr)\,dy,\; y>0,$ so that
its Mellin transform is
\begin{equation}
{\bf E}\bigl[Y^q\bigr] = \Gamma\Bigl(1-\frac{q}{\tau}\Bigr),\,\,\,
\Re(q)<\tau,
\end{equation}
and $\log Y$ is infinitely divisible. Given
$\lambda_i>-1/\tau,$ let $X_1,\,X_2,\,X_3$ have the $\beta^{-1}(\tau, b)$ distribution with the parameters\footnote{
Without loss of generality, $\lambda_2\geq\lambda_1.$ 
If $\lambda_2=\lambda_1,$ then $X_1=1,$ otherwise the parameters of $X_1$ satisfy Theorem \ref{eta22}.
}
\begin{align}
X_1 &\triangleq \beta_{2,2}^{-1}\Bigl(\tau,
b_0=1+\tau+\tau\lambda_1,\,b_1=\tau(\lambda_2-\lambda_1)/2, \,
b_2=\tau(\lambda_2-\lambda_1)/2\Bigr),\\
X_2 & \triangleq \beta_{2,2}^{-1}\Bigl(\tau,
b_0=1+\tau+\tau(\lambda_1+\lambda_2)/2,\,b_1=1/2,\,b_2=\tau/2\Bigr),\\
X_3 & \triangleq \beta_{2,2}^{-1}\Bigl(\tau, b_0=1+\tau,\,
b_1=(1+\tau+\tau\lambda_1+\tau\lambda_2)/2, \,
b_2=(1+\tau+\tau\lambda_1+\tau\lambda_2)/2\Bigr).
\end{align}
\begin{theorem}[Structure of the Selberg integral distribution]\label{B2} Let $\tau=2/\mu.$ $M_{(\mu, \lambda_1, \lambda_2)}$
decomposes into independent factors,
\begin{equation}\label{Decomposition}
M_{(\mu, \lambda_1, \lambda_2)} \overset{{\rm in \,law}}{=} 2\pi\,
2^{-\bigl[3(1+\tau)+2\tau(\lambda_1+\lambda_2)\bigr]/\tau}\,\Gamma\bigl(1-1/\tau\bigr)^{-1}\,
L\,X_1\,X_2\,X_3\,Y.
\end{equation}
In particular, $\log M_{(\mu, \lambda_1, \lambda_2)}$ is absolutely continuous and infinitely divisible. Its 
L$\acute{\text{e}}$vy-Khinchine decomposition 
$\log\mathfrak{M}(q\,|\,\mu,\lambda_1,\lambda_2) =
q\,\mathfrak{m}(\mu) + \frac{1}{2} q^2
\sigma^2(\mu)+\int_{\mathbb{R}\setminus \{0\}} \bigl(e^{q u}-1-q
u/(1+u^2)\bigr) d\mathcal{M}_{(\mu, \lambda_1, \lambda_2)}(u)$
is
\begin{gather}
\sigma^2(\mu) = \frac{4\,\log 2}{\tau}, \\
\mathcal{M}_{(\mu, \lambda_1,
\lambda_2)}(u) =-\int\limits_u^\infty 
\Bigl[\frac{\bigl(e^x+e^{-x\tau\lambda_1}+e^{-x\tau\lambda_2}+e^{-x(1+\tau(1+\lambda_1+\lambda_2))}\bigr)}{\bigl(e^x-1\bigr)(e^{x\tau}-1)}
-\frac{e^{-x(1+\tau(1+\lambda_1+\lambda_2))/2}}{\bigl(e^{x/2}-1\bigr)(e^{x\tau/2}-1)}\Bigr]\,
\frac{dx}{x}
\end{gather}
for $u>0,$ and $\mathcal{M}_{(\mu, \lambda_1, \lambda_2)}(u)=0$ for
$u<0,$ and some constant $\mathfrak{m}(\mu)\in\mathbb{R}.$ 
The Stieltjes moment problems for both positive, cf. \eqref{posmomentformula}, and negative, cf. \eqref{negmomentformula}, moments of $M_{(\mu, \lambda_1, \lambda_2)}$ are indeterminate. 
\end{theorem}
We finally note that the Mellin transform of the Selberg integral distribution has a remarkable factorization, which extends Selberg's finite product of gamma factors to an infinite product.
\begin{theorem}[Factorization of the Mellin transform]
\begin{align}
\mathfrak{M}(q\,|\,\mu,\lambda_1,\lambda_2) = &
\frac{\tau^{q}\Gamma\bigl(1-q/\tau\bigr)\Gamma\bigl(2-2q+\tau(1+\lambda_1+\lambda_2)\bigr)}{
\Gamma^{q}\bigl(1-1/\tau\bigr)\Gamma\bigl(2-q+\tau(1+\lambda_1+\lambda_2)\bigr)}
\prod\limits_{m=1}^\infty \left(m\tau\right)^{2q}
\frac{\Gamma\bigl(1-q+m\tau\bigr)}{\Gamma\bigl(1+m\tau\bigr)}
 \star \nonumber \\
\star &
\frac{\Gamma\bigl(1-q+\tau\lambda_1+m\tau\bigr)}{\Gamma\bigl(1+\tau\lambda_1+m\tau\bigr)}
\frac{\Gamma\bigl(1-q+\tau\lambda_2+m\tau\bigr)}{\Gamma\bigl(1+\tau\lambda_2+m\tau\bigr)}
\frac{\Gamma\bigl(2-q+\tau(\lambda_1+\lambda_2)+m\tau\bigr)}{\Gamma\bigl(2-2q+\tau(\lambda_1+\lambda_2)+m\tau\bigr)}.
\label{ShintaniFactor}
\end{align}
\end{theorem} 

We conjectured in \cite{Me4} in the special case of $\lambda_1=\lambda_2=0$ and in \cite{MeIMRN}
in general that the Selberg integral distribution coincides with the law of the generalized total mass of the limit lognormal measure.
\begin{conjecture}\label{myresult}
Let $\mu\in(0, 2)$ and $\lambda_i>-\mu/2,$ then
\begin{equation}
\int_0^1 s^{\lambda_1} (1-s)^{\lambda_2} 
\,M_\mu(ds) \overset{{\rm in \,law}}{=} 
M_{(\mu, \lambda_1, \lambda_2)}. 
\end{equation}
\end{conjecture}
The rationale for this conjecture is that we constructed a family of probability distributions parameterized by $\mu,$ $\lambda_1,$ and $\lambda_2$ having the properties that (1) its moments match the moments of the (generalized) total mass $\int_0^1 s^{\lambda_1} (1-s)^{\lambda_2} \,M_\mu(ds),$ \emph{i.e.} the Selberg integral and (2) the asymptotic expansion of its Mellin transform
coincides with the intermittency expansion of the total mass. It is finally worth pointing out that the restriction $\lambda_i>-\mu/2$ is artificial and
only imposed so that the Mellin transform can be defined over $\Re(q)<2/\mu.$
It is not difficult to see from \eqref{multifractallaw} that 
\begin{equation}\label{lawlargenumbers}
\frac{\log M_\mu(0,\,t)}{\log t} \rightarrow 1+\frac{\mu}{2}, \,\,\text{a.s.}
\end{equation}
in the limit $t\rightarrow 0$  
so that the integral $\int_{0}^1 s^{\lambda_1} (1-s)^{\lambda_2}\,M_\mu(ds)$ is actually defined for $\lambda_i+\mu/2>-1$ and its Mellin transform for
$\Re(q)<\min\{2/\mu,\,1+(1+\lambda_i)2/\mu\}.$

\section{Calculations}
\noindent In this section we will give heuristic derivations of Conjectures
\ref{Mellinvar} and \ref{Mellin} and their corollaries assuming Conjecture \ref{myresult} to be true. The derivations are based on certain exact calculations, combined with a key assumption in the case of Conjecture
\ref{Mellin} that is explained below.

The main idea of the derivation of Conjecture 
\ref{Mellinvar} is that the $S^{(1)}_t(\mu,u,\varepsilon)$ and $S^{(2)}_t(\mu,u,\varepsilon)$ statistics that we introduced in \eqref{Su} in Section 2 converge in law, up to $O(\varepsilon)$ 
terms, to (modifications of) the processes $\omega_{\mu, \varepsilon}(u)-\omega_{\mu, \varepsilon}(0)$ and
$\omega_{\mu, \varepsilon}(u),$ respectively, in the limit $t\rightarrow \infty,$ where $\omega_{\mu, \varepsilon}(u)$ is the approximation of the gaussian free field that we defined in \eqref{mean}-\eqref{cov2} (recall $L=1$). 
\begin{lemma}[Main lemma for $S^{(1)}_t(\mu,u,\varepsilon)$]\label{main}
Let $f^{(1)}_{\varepsilon, u}(x)$ be as in \eqref{fu}, $\chi^{(1)}_u(x)=\chi_{[0,\,u]}(x),$
\begin{equation}
f^{(1)}_{\varepsilon, u}(x) \triangleq (\chi^{(1)}_u\star\phi_\varepsilon)(x) = \frac{1}{\varepsilon}\int \chi^{(1)}_u(x-y)\phi(y/\varepsilon)\,dy,
\end{equation}
$\kappa$ be as in \eqref{kappa}, and the scalar product be defined as in \eqref{scalar}. Then, given $0<u, v<1,$ in the limit $\varepsilon\rightarrow 0,$
\begin{align}
\langle f^{(1)}_{\varepsilon, u},\,f^{(1)}_{\varepsilon, v}\rangle & = -\frac{1}{2\pi^2} 
\bigl(\log\varepsilon - \kappa + \log|u-v| - \log|u| - \log|v| \bigr)+ O(\varepsilon), \; \text{if $|u-v|\gg \varepsilon$}, \label{covar}
 \\
\langle f^{(1)}_{\varepsilon, u},\,f^{(1)}_{\varepsilon, u}\rangle & =
 -\frac{1}{\pi^2} 
\bigl(\log\varepsilon - \kappa - \log|u| \bigr)+ O(\varepsilon),
 \; \text{if $u=v.$} \label{var}
\end{align}
\end{lemma}
\begin{proof}
Recall that $f'^{(1)}_{\varepsilon, u}=\phi_\varepsilon(x)-\phi_\varepsilon(x-u),$ where $\phi_\varepsilon(x)\triangleq 1/\varepsilon\phi(x/\varepsilon)$ is the rescaled bump function, cf. \eqref{fuprime1}. Then,
\begin{equation}\label{klog}
\int \phi_\varepsilon(x)\log|x-z| dx =
\begin{cases}
& \log|z| +O(\varepsilon^2), \;\text{if $|z| \gg \varepsilon,$} \\ 
& \log \varepsilon + \int \phi(x)\log\big|x-\frac{z}{\varepsilon}\big|dx, \;\text{if $|z|<\varepsilon$}.
\end{cases}
\end{equation}
It follows 
\begin{equation}\label{kklog}
\int \big(\phi_\varepsilon(x) - \phi_\varepsilon(x-u)\bigr)\log|x-y| dx =
\begin{cases}
& \log|y|-\log|u-y|+O(\varepsilon^2)\;\text{if $|y|,\,|y-u|\gg\varepsilon$}, \\
& \log\varepsilon + \int \phi(x)\log\big|x-\frac{y}{\varepsilon}\big|dx 
-\log|u-y|+O(\varepsilon^2)\;\text{if $|y|<\varepsilon$},\\
& \log|y|-\log\varepsilon - \int \phi(x)\log\big|x-\frac{y-u}{\varepsilon}\big|dx 
+O(\varepsilon^2)
\end{cases}
\end{equation}
if $|y-u|<\varepsilon.$
Now, by the definition of the scalar product,
\begin{align}
\langle f^{(1)}_{\varepsilon, u},\,f^{(1)}_{\varepsilon, v}\rangle & = -\frac{1}{2\pi^2}
\int\bigl(\phi_\varepsilon(y) - \phi_\varepsilon(y-v) \bigr)\bigl( \phi_\varepsilon(x) - \phi_\varepsilon(x-u)   \bigr) \log|x-y|dxdy, \nonumber \\
& =  -\frac{1}{2\pi^2}
\int\phi_\varepsilon(y)\Bigl[ \bigl(\phi_\varepsilon(x) - \phi_\varepsilon(x-u)   \bigr) \log|x-y|dx\Bigr]\,dy \nonumber \\
& + \frac{1}{2\pi^2}
\int\phi_\varepsilon(y)\Bigl[ \bigl(\phi_\varepsilon(x) - \phi_\varepsilon(x-u)   \bigr) \log|x-(y+v)|dx\Bigr]\,dy. 
\end{align}
The resulting integrals are all of the functional form that we treated in \eqref{klog} and \eqref{kklog} so the result follows. \qed
\end{proof}
\begin{corollary}[Covariance structure of $S^{(1)}_t(\mu, u, \varepsilon)$]\label{covstructureS}
Let the statistic $S^{(1)}_t(\mu, u, \varepsilon)$ be as in \eqref{Su}, 
\begin{equation}
S^{(1)}_t(\mu, u, \varepsilon)\triangleq \pi\sqrt{2\mu}\Bigl[\sum\limits_{\gamma} f^{(1)}_{\varepsilon, u}\bigl(\gamma(t)\bigr) - \frac{\log t}{2\pi \lambda(t)} \int f^{(1)}_{\varepsilon, u}(x) dx\Bigr].
\end{equation}
Then, in the limit $t\rightarrow\infty$ the process $u\rightarrow S^{(1)}_t(\mu, u, \varepsilon),$
$u\in (0, 1),$
converges in law to the centered gaussian field $S^{(1)}(\mu, u, \varepsilon)$ having the asymptotic covariance
\begin{align}\label{CovS}
{\bf Cov}\Bigl( S^{(1)}(\mu, u, \varepsilon), \,S^{(1)}(\mu, v, \varepsilon)  \Bigr)
 = &
\begin{cases}
& -\mu 
\bigl(\log\varepsilon - \kappa + \log|u-v| - \log|u| - \log|v| \bigr), \; \text{if $|u-v|\gg \varepsilon$}, 
 \\
 &  -2\mu 
\bigl(\log\varepsilon - \kappa - \log|u| \bigr),
 \; \text{if $u=v,$}
\end{cases} \nonumber \\
 & + O(\varepsilon).
\end{align}
\end{corollary}
\begin{proof}
This follows from Theorem \ref{strong} and the calculation of the scalar 
product in Lemma \ref{main}.
\qed
\end{proof}
\begin{lemma}[Main lemma for $S^{(2)}_t(\mu,u,\varepsilon)$]\label{main2}
Let $f^{(2)}_{\varepsilon, u}(x)$ be as in \eqref{fu}, $\chi^{(2)}_u(x)=\chi_{[-1/\varepsilon,\,u]}(x).$
Then, given $0<u, v<1,$ in the limit $\varepsilon\rightarrow 0,$
\begin{align}
\langle f^{(2)}_{\varepsilon, u},\,f^{(2)}_{\varepsilon, v}\rangle & = -\frac{1}{2\pi^2} 
\bigl(3\log\varepsilon - \kappa + \log|u-v| \bigr)+ O(\varepsilon), \; \text{if $|u-v|\gg \varepsilon$}, \label{covar2}
 \\
\langle f^{(2)}_{\varepsilon, u},\,f^{(2)}_{\varepsilon, u}\rangle & =
 -\frac{1}{2\pi^2} 
\bigl(4\log\varepsilon - 2\kappa \bigr)+ O(\varepsilon),
 \; \text{if $u=v.$} \label{var2}
\end{align}
\end{lemma}
\begin{proof}
The calculation is the same as in the proof of Lemma \ref{main} and will be omitted.\qed
\end{proof}
\begin{corollary}[Covariance structure of $S^{(2)}_t(\mu, u, \varepsilon)$]\label{covstructureS2}
Let the statistic $S^{(2)}_t(\mu, u, \varepsilon)$ be as in \eqref{Su}, corresponding to $f^{(2)}_{\varepsilon, u}.$
Then, in the limit $t\rightarrow\infty$ the process $u\rightarrow S^{(2)}_t(\mu, u, \varepsilon),$
$u\in (0, 1),$
converges in law to the centered gaussian field $S^{(2)}(\mu, u, \varepsilon)$ having the asymptotic covariance
\begin{align}\label{CovS2}
{\bf Cov}\Bigl( S^{(2)}(\mu, u, \varepsilon), \,S^{(2)}(\mu, v, \varepsilon)  \Bigr)
 = &
\begin{cases}
 -\mu 
\bigl(3\log\varepsilon - \kappa + \log|u-v| \bigr), \; \text{if $|u-v|\gg \varepsilon$}, 
 \\
   -\mu 
\bigl(4\log\varepsilon - 2\kappa  \bigr),
 \; \text{if $u=v,$}
\end{cases} \nonumber \\
 & + O(\varepsilon).
\end{align}
\end{corollary}
\begin{proof}
This follows from Theorem \ref{strong} and the calculation of the scalar 
product in Lemma \ref{main2}.
\qed
\end{proof}

On the other hand, it is elementary to see from \eqref{mean}-\eqref{cov2}
that the gaussian free field at scale $\varepsilon$ satisfies
\begin{equation}
{\bf Cov}\Bigl( \omega_{\mu, \varepsilon}(u), \omega_{\mu, \varepsilon}(v)  \Bigr)
 =
\begin{cases}
& -\mu 
\log|u-v|, \; \text{if $|u-v| > \varepsilon$}, 
 \\
 &  \mu 
\bigl(1-\log\varepsilon\bigr), \; \text{if $u=v,$}
\end{cases}
\end{equation}
so that the centered gaussian free field at scale $\varepsilon,$
\begin{equation}\label{centeredomega}
\bar{\omega}_{\mu, \varepsilon}(u) \triangleq \omega_{\mu, \varepsilon}(u)-\omega_{\mu, \varepsilon}(0),
\end{equation}
satisfies
\begin{equation}
{\bf Cov}\Bigl( \bar{\omega}_{\mu, \varepsilon}(u), \bar{\omega}_{\mu, \varepsilon}(v)  \Bigr)
 =
\begin{cases}
& -\mu 
\bigl(\log\varepsilon - 1 + \log|u-v| - \log|u| - \log|v| \bigr), \; \text{if $|u-v| > \varepsilon$}, 
 \\
 &  -2\mu 
\bigl(\log\varepsilon - 1 - \log|u| \bigr), \; \text{if $u=v.$}
\end{cases}
\end{equation}
As we remarked in our discussion of \eqref{Msum}, the choice of truncation,
and, in particular, the choice of the constant, $\varepsilon-$independent term
in the covariance of the free field,  
has no effect on the law of the total mass so long as the key normalization condition in \eqref{normalization} holds. Hence, for our purposes, we can re-define $\omega_{\mu, \varepsilon}(u)$ to be
\begin{align}
{\bf{E}} \left[\omega_{\mu, \varepsilon}(u)\right] & =  -\frac{\mu}{2} \,
\left(\kappa-\log\varepsilon\right), \label{meank} \\
{\bf{Cov}}\left[\omega_{\mu, \varepsilon}(u), \,\omega_{\mu, \varepsilon}
(v)\right] &  =
\begin{cases}\label{covk}
 -
\mu \, \log|u-v|, \, \varepsilon\leq|u-v|\leq 1,  \\
 \mu
\left(\kappa-\log\varepsilon\right),
\end{cases}
\end{align}
so that $\bar{\omega}_{\mu, \varepsilon}(u)$ has asymptotically the same covariance as that of $S^{(1)}(\mu, u, \varepsilon)$ 
in \eqref{CovS}. Hence,
\begin{equation}\label{S1approxlaw}
S^{(1)}(\mu, u, \varepsilon)  \overset{{\rm in \,law}}{\approx} \bar{\omega}_{\mu, \varepsilon}(u) 
\end{equation}
as stochastic processes in the limit $\varepsilon\rightarrow 0,$ up to zero-mean corrections having covariance of the order $ O(\varepsilon).$ 
Similarly, by comparing covariances of $S^{(2)}(\mu, u, \varepsilon)$ in \eqref{CovS2} and $\omega_{\mu, \varepsilon}(u)$ in \eqref{covk} and recalling that $S^{(2)}(\mu, u, \varepsilon)$ is centered, while the mean of $\omega_{\mu, \varepsilon}(u)$
is given in \eqref{meank}, we obtain, also up to $O(\varepsilon),$
\begin{equation}\label{S2approxlaw}
S^{(2)}(\mu, u, \varepsilon)  \overset{{\rm in \,law}}{\approx} \omega_{\mu, \varepsilon}(u) + \mathcal{N}\Bigl(
-\frac{\mu}{2}(\log\varepsilon-\kappa), -\mu(3\log\varepsilon-\kappa) \Bigr),
\end{equation}
where $\mathcal{N}\bigl(
-(\mu/2)(\log\varepsilon-\kappa), -\mu(3\log\varepsilon-\kappa) \bigr)$ is an independent gaussian random variable
having the mean $-(\mu/2)(\log\varepsilon-\kappa)$ and variance $-\mu(3\log\varepsilon-\kappa).$

We can now explain the origin of Conjecture \ref{Mellinvar}. The basic idea is to form the exponential functional of the statistic $S^{(i)}_t(\mu, u, \varepsilon)$ and compute its Mellin transform by analogy with the gaussian free field so as to obtain the total mass of the limit lognormal measure in the limit. The principal obstacle in the first case is that 
$S^{(1)}_t(\mu, u, \varepsilon)$ behaves like $\bar{\omega}_{\mu, \varepsilon}(u)$ as opposed to
$\omega_{\mu, \varepsilon}(u)$ so that its exponential functional does not exist
in the limit $\varepsilon\rightarrow 0.$ This obstacle is overcome
by appropriately rescaling the Mellin transform as shown in the following
theorem. To this end, we will first formulate a general proposition and then specialize it to Conjecture \ref{Mellinvar}. In what follows $\varphi(u)$ can be a general test function, however for clarity, we will restrict ourselves to 
\begin{equation}\label{varphi}
\varphi(u) \triangleq u^{\lambda_1}(1-u)^{\lambda_2},
\;\lambda_1,\,\lambda_2>-\frac{\mu}{2}.
\end{equation}
\begin{theorem}[Rescaled Mellin transforms]\label{general}
Let  $0<\mu<2$ and $I$ be a subinterval of the unit interval. Then, for $\Re(q)<2/\mu$ we have
\begin{equation}
\lim\limits_{\varepsilon\rightarrow 0} e^{\mu(\log \varepsilon-\kappa)\frac{q(q+1)}{2}} 
{\bf E} \Bigl[\Bigl(\int_{I}  u^{-\mu q} \,\varphi(u)\,e^{\bar{\omega}_{\mu, \varepsilon}(u)} du\Bigr)^q\Bigr] =  {\bf E} \Bigl[ \Bigl(\int_{I} \varphi(u)\, M_\mu(du)\Bigr)^q\Bigr]. \label{gen2}
\end{equation} 
For $-(1+\lambda_1)/\mu-1/2<\Re(q)<2/\mu$ we have
\begin{equation}
\lim\limits_{\varepsilon\rightarrow 0} e^{\mu(\log \varepsilon-\kappa)\frac{q(q+1)}{2}} 
{\bf E} \Bigl[\Bigl(\int_I \varphi(u)\,e^{\bar{\omega}_{\mu, \varepsilon}(u)} du\Bigr)^q\Bigr] =  {\bf E} \Bigl[ \Bigl(\int_{I} u^{\mu q}\,\varphi(u)\, M_\mu(du)\Bigr)^q\Bigr]. \label{gen1}
\end{equation}
Moreover, the positive integral moments of order $n<2/\mu$ satisfy
\begin{align}
\lim\limits_{\varepsilon\rightarrow 0} e^{\mu(\log \varepsilon-\kappa)\frac{n(n+1)}{2}}
{\bf E} \Bigl[\int\limits_{I}  u^{-\mu n}\, \varphi(u)\,e^{\bar{\omega}_{\mu, \varepsilon}(u)}
\,du\Bigr]^n & = \int\limits_{I^n} \prod_{i=1}^n
\bigl[\varphi(s_i)\bigr] \prod\limits_{i<j}^n
|s_i-s_j|^{-\mu} ds_1\cdots ds_n, \\
\lim\limits_{\varepsilon\rightarrow 0} e^{\mu(\log \varepsilon-\kappa)\frac{n(n+1)}{2}}
{\bf E} \Bigl[\int\limits_{I} \varphi(u)\,e^{\bar{\omega}_{\mu, \varepsilon}(u)}
\,du\Bigr]^n & = \int\limits_{I^n} \prod_{i=1}^n
\bigl[\varphi(s_i)\,s_i^{\mu\,n}\bigr] \prod\limits_{i<j}^n
|s_i-s_j|^{-\mu} ds_1\cdots ds_n.
\end{align}
\end{theorem}
The proof requires two auxiliary calculations. The first involves a version of Girsanov-type theorem for gaussian fields, which extends what we used
in our original derivation of intermittency differentiation, cf. Section 8 of \cite{Me2}. For concreteness, we state it here for the $\omega_{\mu, \varepsilon}(t)$ process. 
\begin{lemma}[Girsanov]\label{Girsanov}
Let 
$I$ be a subinterval of the unit interval.  Then, we have for $\Re(p)<2/\mu,$
\begin{equation}\label{girs1}
\lim\limits_{\varepsilon\rightarrow 0} e^{\mu(\log\varepsilon-\kappa)\frac{q(q+1)}{2}} 
{\bf E}\Bigl[e^{-q\,\omega_{\mu,\varepsilon}(0)}\Bigl(\int_{I}  u^{-\mu q} \,\varphi(u)\,e^{\omega_{\mu,\varepsilon}(u)}du\Bigr)^p\Bigr] =  {\bf E} \Bigl[ \Bigl(\int_{I}\varphi(u)\,M_\mu(du)\Bigr)^p \Bigr],
\end{equation}
and for $\Re(q)>-(1+\lambda_1)/\mu-1/2,$ $\Re(p)<\min\{2/\mu, \,2/\mu(1+\lambda_1)+1+2\Re(q)\},$
\begin{equation}\label{girs2}
\lim\limits_{\varepsilon\rightarrow 0} e^{\mu(\log\varepsilon-\kappa)\frac{q(q+1)}{2}} 
{\bf E}\Bigl[e^{-q\,\omega_{\mu,\varepsilon}(0)}\Bigl(\int_{I} \varphi(u)\,e^{\omega_{\mu,\varepsilon}(u)}du\Bigr)^p\Bigr] =  {\bf E} \Bigl[ \Bigl(\int_{I} u^{\mu q} \,\varphi(u)\,M_\mu(du)\Bigr)^p \Bigr].
\end{equation}
\end{lemma}
\begin{proof} 
The integral $\int_{I} u^{\mu q}\,\varphi(u)\,M_\mu(du)$ is defined for $\mu\,\Re(q)+\lambda_1+\mu/2>-1$ and its Mellin transform for
$\Re(p)<\min\{2/\mu, \,2/\mu(1+\lambda_1)+1+2\Re(q)\},$ cf. \eqref{lawlargenumbers}.
Next, we will establish the identities
\begin{align}
e^{\mu(\log\varepsilon-\kappa) \frac{q(q+1)}{2}}\,{\bf E}\Bigl[e^{-q\omega_{\mu,\varepsilon}(0)}\Bigl(\int_{I} u^{-\mu q} \,\varphi(u)\,e^{\omega_{\mu,\varepsilon}(u)}du\Bigr)^p\Bigr] & =  {\bf E} \Bigl[ \Bigl(\int_{I} \varphi(u)\,e^{\omega_{\mu,\varepsilon}(u)}du\Bigr)^p \Bigr],  \label{girsid2} \\
e^{\mu(\log\varepsilon-\kappa) \frac{q(q+1)}{2}}\,{\bf E}\Bigl[e^{-q\omega_{\mu,\varepsilon}(0)}\Bigl(\int_{I} \varphi(u)\,e^{\omega_{\mu,\varepsilon}(u)}du\Bigr)^p\Bigr] & =  {\bf E} \Bigl[ \Bigl(\int_{I} u^{\mu q} \,\varphi(u)\,e^{\omega_{\mu,\varepsilon}(u)}du\Bigr)^p \Bigr]. \label{girsid}
\end{align}
This result is equivalent to a change of measure. Introduce an equivalent probability measure
\begin{equation}
d\mathcal{Q}\triangleq e^{-\mu(\kappa-\log\varepsilon) \frac{q(q+1)}{2}} e^{-q\omega_{\mu, \varepsilon}(0)} \,d\mathcal{P},
\end{equation}
where $\mathcal{P}$ is the original probability measure
corresponding to ${\bf E}.$ In fact, it follows from \eqref{meank} and \eqref{covk}
that 
\begin{equation}
{\bf E} \Bigl[e^{-\mu(\kappa-\log\varepsilon) \frac{q(q+1)}{2}} e^{-q\omega_{\mu, \varepsilon}(0)}\Bigr] = 1.
\end{equation}
Then, the law of the process
$u\rightarrow \omega_{\mu, \varepsilon}(u)-q\,{\bf Cov}\bigl(\omega_{\mu, \varepsilon}(u),\,\omega_{\mu, \varepsilon}(0)\bigr)$ with respect to $\mathcal{P}$ equals the law of the original process $u\rightarrow \omega_{\mu, \varepsilon}(u)$ with respect to $\mathcal{Q}.$  Indeed, it is easy to show that the two
processes have the same finite-dimensional distributions by
computing their characteristic functions. The continuity of sample paths can then be used to conclude that the equality of all finite-dimensional
distributions implies the equality in law. Once this equality is established, then
\begin{align}
e^{-\mu(\kappa-\log\varepsilon) \frac{q(q+1)}{2}}{\bf E}\Bigl[e^{-q\omega_{\mu,\varepsilon}(0)}\Bigl(\int_{I} u^{-\mu q} \,
\varphi(u)\,e^{\omega_{\mu,\varepsilon}(u)}du\Bigr)^p\Bigr] & =  {\bf E}\Big\vert_{\mathcal{Q}} \Bigl[ 
\Bigl(\int_{I}  u^{-\mu q} \,\varphi(u)\,e^{\omega_{\mu,\varepsilon}(u)}du\Bigr)^p \Bigr], \nonumber \\
& =  {\bf E} \Bigl[ 
\Bigl(\int_{I} \varphi(u)\,e^{\omega_{\mu,\varepsilon}(u)}du\Bigr)^p \Bigr], 
\end{align}
which completes the proof of \eqref{girsid2} and of \eqref{girs1}
by taking the limit $\varepsilon\rightarrow 0.$ The argument for
\eqref{girs2} is essentially the same except that one needs to restrict
$p$ so that the right-hand side is well-defined. Given \eqref{Mfunc}, the conditions are $\Re(p)<2/\mu$ and $1-\Re(p)+2\Re(q)+(1+\lambda_1)2/\mu>0.$ 
\qed
\end{proof}

The second step in the proof of Theorem \ref{general} entails a key moment
calculation, in which we compute the asymptotic of the rescaled 
positive integral moments in terms of generalized Selberg integrals.
\begin{lemma}[Single moments]\label{single}
Let $n<2/\mu$ and $I$ be a subinterval of the unit interval. Then, as $\varepsilon\rightarrow 0,$ 
\begin{align}
{\bf E} \Bigl[\int\limits_{I} u^{-\mu n}\,\varphi(u)\,e^{\bar{\omega}_{\mu, \varepsilon}(u)}
\,du\Bigr]^n & \thicksim e^{\mu(-\log\varepsilon+\kappa)\frac{n(n+1)}{2}} 
\int\limits_{I^n} \prod_{i=1}^n
\bigl[\varphi(s_i)\bigr] \prod\limits_{i<j}^n
|s_i-s_j|^{-\mu} ds_1\cdots ds_n, \label{mom2} \\
{\bf E} \Bigl[\int\limits_{I} \varphi(u)\,e^{\bar{\omega}_{\mu, \varepsilon}(u)}
\,du\Bigr]^n & \thicksim e^{\mu(-\log\varepsilon+\kappa)\frac{n(n+1)}{2}} 
\int\limits_{I^n} \prod_{i=1}^n
\bigl[\varphi(s_i)s_i^{\mu\,n}\bigr] \prod\limits_{i<j}^n
|s_i-s_j|^{-\mu} ds_1\cdots ds_n. \label{mom1} 
\end{align} 
\end{lemma}
\begin{proof}
We will give the proof of \eqref{mom1}.
The idea is to discretize as in \eqref{Msum}. 
Let $N=|I|/\varepsilon,$ $\bar{\omega}_j\triangleq \bar{\omega}_{\mu, \varepsilon}(s_j),$
$s_j = \varepsilon j,$ $j=1\cdots N.$ As $\bar{\omega_1},\cdots, \bar{\omega}_N$
are jointly gaussian with zero mean, we have
\begin{equation}\label{intermedcomp}
{\bf{E}} \bigl[\varepsilon\sum_{j=1}^N \varphi(s_j)\,
e^{\bar{\omega}_j}\bigr]^n = \varepsilon^n \sum_{j_1\ldots j_n=1}^N
\varphi(s_{j_1})\cdots \varphi(s_{j_n}) e^{\frac{1}{2}{\bf Var}
\left(\bar{\omega}_{j_1}+\ldots+\bar{\omega}_{j_n}\right)}.
\end{equation}
Now, using \eqref{meank} and \eqref{covk},
\begin{align}
{\bf Var}
\left(\bar{\omega}_{j_1}+\ldots+\bar{\omega}_{j_n}\right) & = \sum\limits_{k=1}^n {\bf Var} \bar{\omega}_{j_k} + 2\sum\limits_{k<l=1}^n {\bf Cov}\bigl (
\bar{\omega}_{j_k}, \bar{\omega}_{j_l}  \bigr) =  2\mu\sum\limits_{k=1}^n\bigl(-\log\varepsilon+\kappa + \log s_{j_k}\bigr) + \nonumber \\
& + 
2\mu \sum\limits_{k<l=1}^n \bigl(-\log\varepsilon+\kappa-\log|s_{j_k}-s_{j_l}|+
\log|s_{j_k}|+\log|s_{j_l}|\bigr), \nonumber \\
& = \mu(-\log\varepsilon+\kappa)n(n+1) + 2\mu n \sum\limits_{k=1}^n\log s_{j_k} - 2\mu\sum\limits_{k<l=1}^n \log|s_{j_k}-s_{j_l}|.
\end{align}
It follows from \eqref{intermedcomp} that
\begin{equation} {\bf{E}}
\bigl[\varepsilon\sum_{j=1}^N \varphi(s_j) e^{\bar{\omega}_j}\bigr]^n
 \thicksim  e^{\mu(-\log\varepsilon+\kappa)\frac{n(n+1)}{2}}\,\varepsilon^n\sum_{j_1\ldots j_n=1}^N \prod\limits_{k=1}^n \bigl[\varphi(s_{j_k})
s_{j_k}^{\mu n}\bigr] 
\prod_{k<l=1}^n |s_{j_k}-s_{j_l}|^{-\mu}.
\end{equation}
In the limit we have $\varepsilon \sum_{j_k=1}^N \rightarrow
\int_I ds_k,$ hence the result. The proof of \eqref{mom2}
follows by re-labeling $\varphi(u)\rightarrow u^{-\mu n} \varphi(u).$
\qed
\end{proof}
We note that the same type of argument gives a formula for the joint moments.
\begin{lemma}[Joint moments]\label{joint}
Let $N=n+m,$ $N<2/\mu,$ 
and $I_1,I_2$ be subintervals of $[0, 1].$ Then, as $\varepsilon\rightarrow 0,$ 
\begin{align}
{\bf E} \Bigl[\Bigl(\int\limits_{I_1} \varphi_1(u)\,e^{\bar{\omega}_{\mu, \varepsilon}(u)}
\,du\Bigr)^n \Bigl(\int\limits_{I_2} \varphi_2(u)\,e^{\bar{\omega}_{\mu, \varepsilon}(u)}
\,du\Bigr)^m\Bigr] & \thicksim e^{\mu(-\log\varepsilon+\kappa)\frac{N(N+1)}{2}} 
\int\limits_{I_1^n\times I_2^m} \prod_{i=1}^n
\bigl[\varphi_1(s_i)s_i^{\mu\,N}\bigr] \nonumber \\ & \times \prod_{i=n+1}^N
\bigl[\varphi_2(s_i)s_i^{\mu\,N}\bigr]\prod\limits_{i<j}^N
|s_i-s_j|^{-\mu} ds_1\cdots ds_N.\label{jointmomentsformula}
\end{align} 
\end{lemma}

We can now give a proof of Theorem \ref{general} and explain the origin of Conjecture \ref{Mellinvar}.
\begin{proof} {\bf of Theorem \ref{general}.}
Let $\Re(p)<2/\mu.$ 
By Lemma \ref{Girsanov},
\begin{equation}\label{qid}
\lim\limits_{\varepsilon\rightarrow 0} e^{\mu(\log\varepsilon-\kappa) \frac{q(q+1)}{2}} 
{\bf E}\Bigl[e^{-q\omega_{\mu,\varepsilon}(0)}\Bigl(\int_{I} u^{-\mu q}\,\varphi(u)\, e^{\omega_{\mu,\varepsilon}(u)}du\Bigr)^p\Bigr] = {\bf E} \Bigl[ \Bigl(\int_{I} \varphi(u)\, M_\mu(du)\Bigr)^p\Bigr].
\end{equation}
In particular, \eqref{qid} holds for $q=p,$ if $\Re(q)<2/\mu.$
On the other hand,
the left-hand side of \eqref{qid} can be reduced when $q=p$ as follows.
Substituting the definition of $\bar{\omega}_{\mu,\varepsilon}(u),$ cf. \eqref{centeredomega}, we obtain
\begin{align}
\lim\limits_{\varepsilon\rightarrow 0} e^{\mu(\log\varepsilon-\kappa) \frac{q(q+1)}{2}} &
{\bf E}\Bigl[e^{-q\omega_{\mu,\varepsilon}(0)}\Bigl(\int_{I} u^{-\mu q}\,\varphi(u)\,e^{\omega_{\mu,\varepsilon}(u)}du\Bigr)^q\Bigr], \nonumber \\
& = \lim\limits_{\varepsilon\rightarrow 0}  e^{\mu(\log\varepsilon-\kappa) \frac{q(q+1)}{2}} {\bf E} \Bigl[\Bigl(\int_{I} u^{-\mu q}\,\varphi(u)\,e^{\bar{\omega}_{ \mu,\varepsilon}(u)}du\Bigr)^q\Bigr]. 
\end{align}
The argument for \eqref{gen1} is essentially the same except that we use 
\eqref{girs2} instead of \eqref{qid}. Given $\Re(q)>-(1+\lambda_1)/\mu-1/2,$ $\Re(p)<\min\{2/\mu, \,2/\mu(1+\lambda_1)+1+2\Re(q)\},$ 
if $q=p$ the condition $\Re(p)<\min\{2/\mu, \,2/\mu(1+\lambda_1)+1+2\Re(q)\}$ is equivalent to $\Re(q)>-(1+\lambda_1)2/\mu-1$ and is so satisfied.  
\qed
\end{proof}
\begin{proof} {\bf of Conjecture \ref{Mellinvar}.} Let $\varphi(u) =1$ and $I=[0,\,1]$ in Theorem \ref{general}. Recalling the asymptotic equality of laws of $\bar{\omega}_{ \mu,\varepsilon}(u)$ and $S^{(1)}_t(\mu, u, \varepsilon)$ at finite $\varepsilon>0$ in the limit $t\rightarrow \infty$ that we established in \eqref{S1approxlaw} above, we formally 
interchange the order of $t$ limit and $u$ integration, take the $t$ limit by Corollary \ref{covstructureS}, then replace  $S^{(1)}(\mu, u, \varepsilon)$ with $\bar{\omega}_{ \mu,\varepsilon}(u)$ by \eqref{S1approxlaw}, 
and, finally take the limit $\varepsilon\rightarrow 0$ by Theorem \ref{general}. The right-hand sides of these equations 
are known by Conjecture \ref{myresult} and described in \eqref{Mfunc}, with $\lambda_1=\lambda_2=0,$ and  
$\lambda_1=\mu q,$ $\lambda_2=0,$ respectively, cf. Theorem \ref{BSM} above, hence \eqref{M2} and \eqref{M1}. For example, the argument for \eqref{M2} is as follows.
\begin{align}
&\lim\limits_{\varepsilon\rightarrow 0} e^{\mu(\log \varepsilon-\kappa)\frac{q(q+1)}{2}} \Bigl[\lim\limits_{t\rightarrow\infty} 
{\bf E} \Bigl[\Bigl(\int_0^1 u^{-\mu q} e^{S^{(1)}_t(\mu,u,\varepsilon)} du\Bigr)^q\Bigr]\Bigr], \nonumber \\ 
&= \lim\limits_{\varepsilon\rightarrow 0} e^{\mu(\log \varepsilon-\kappa)\frac{q(q+1)}{2}} \Bigl[
{\bf E} \Bigl[\Bigl(\int_0^1 u^{-\mu q} e^{S^{(1)}(\mu,u,\varepsilon)} du\Bigr)^q\Bigr]\Bigr], \nonumber \\ 
& = \lim\limits_{\varepsilon\rightarrow 0} e^{\mu(\log \varepsilon-\kappa)\frac{q(q+1)}{2}} \Bigl[
{\bf E} \Bigl[\Bigl(\int_0^1 u^{-\mu q} e^{\bar{\omega}_{\mu,\varepsilon}(u)} du\Bigr)^q\Bigr]\Bigr], \nonumber \\ 
& = {\bf E} \Bigl[ \Bigl(\int_0^1 M_\mu(du)\Bigr)^q\Bigr].
\end{align}

The derivation of the weak conjecture for $S^{(2)}_t(\mu, u, \varepsilon)$ follows the same steps but is simpler as it is based directly on the asymptotic equality of laws of $\omega_{\mu,\varepsilon}(u)$ and $S^{(2)}_t(\mu, u, \varepsilon)$ at finite 
$\varepsilon>0$ in the limit $t\rightarrow \infty,$ cf. \eqref{S2approxlaw} above. By taking the exponential functional
of $S^{(2)}_t(\mu, u, \varepsilon)$ as in \eqref{M2transfstrong} and using \eqref{S2approxlaw}, we obtain the
Mellin transform of the total mass of the limit lognormal measure times the Mellin transform of $\exp\bigl(\mathcal{N}(
-(\mu/2)(\log\varepsilon-\kappa), -\mu(3\log\varepsilon-\kappa))\bigr).$ The latter accounts for the scaling factor
in \eqref{M2transfstrong}, and the result then follows from Conjecture \ref{myresult} with $\lambda_1=\lambda_2=0.$ 
Similarly, \eqref{M1transfstrong} follows in the same way from Conjecture \ref{myresult} with 
$\lambda_1=\mu q,$ $\lambda_2=0.$
\qed
\end{proof}

We next proceed to the derivation of Conjecture \ref{Mellin}. The basic idea is
to let $\varepsilon(t)$ approach zero ``slowly'' compared to the growth of $t$ to infinity so that the statistic $S^{(i)}_t(\mu, u)$
in \eqref{Sut} behaves as the centered gaussian free field for $i=1$ or the gaussian free field plus an independent gaussian for $i=2$ at the scale $\varepsilon(t).$ The main technical challenge of quantifying the required
rate of decay of $\varepsilon(t)$ is that neither Theorem
\ref{strong} nor \ref{divergvar} applies to the statistic $S^{(i)}_t(\mu, u).$
Theorem \ref{strong} does not apply as the variance of $S^{(i)}_t(\mu, u)$ is divergent in the limit $t\rightarrow \infty .$ Theorem \ref{divergvar} does not apply as the asymptotic of the variance $\int_{-\log t/\lambda(t)}^{\log t/\lambda(t)} |w| |\widehat{\chi_u}(w)|^2\,dw\propto \log (\log t/\lambda(t))$ in \eqref{limitvar} is different from the asymptotic of our variance
$\int\limits_{-\infty}^{\infty} |w| |\widehat{(\chi_u\star\phi_\varepsilon)}(w)|^2\,dw\propto -\log \varepsilon,$
cf. Lemmas \ref{main} and \ref{main2} above.
Instead, we need a slight modification of Theorem \ref{divergvar}. Lemma \ref{modified} shows
that for a sufficiently slowly decaying $\varepsilon(t)$ one can obtain a limiting gaussian field having $\varepsilon(t)$-dependent asymptotic covariance. 
Lemma \ref{limit} shows further that the limiting covariance can be approximated by the scalar product in Theorem \ref{strong}. 
Let $S_t(f)$ be as in \eqref{St}. For simplicity,
we restrict ourselves here to finite linear combinations of indicator functions
$\sum c_k \chi_{u_k},$ $0<u_k<1.$
\begin{lemma}[Modified convergence]\label{modified}
Let $\lambda(t)$ satisfy \eqref{lt} and $\varepsilon(t)$ satisfy \eqref{vart} for $i=1$ or \eqref{vart2} for $i=2,$ respectively.
Let $\chi^{(1)}_u(x)=\chi_{[0,\,u]}(x)$ and $\chi^{(2)}_u(x)=\chi_{[-1/\varepsilon(t),\,u]}(x).$ Define
\begin{equation}\label{mylimitvar}
\sigma_t^{(i)}{}^2 \triangleq \int\limits_{-\log t/\lambda(t)}^{\log t/\lambda(t)} |w| \Big|\sum c_k \widehat{ (\chi^{(i)}_{u_k}\star\phi_{\varepsilon(t)})}(w)\Big|^2\,dw,
\end{equation}
then, as $t\rightarrow \infty,$ 
\begin{equation}\label{limitofSt}
S_t\Bigl(\sum c_k \chi^{(i)}_{u_k}\star\phi_{\varepsilon(t)}\Bigr)  \overset{{\rm in \,law}}{=} \sigma_t^{(i)} \,Y^{(i)}_t + o(1), \; 
Y^{(i)}_t  \overset{{\rm in \,law}}{\rightarrow} \mathcal{N}(0, 1).
\end{equation}
\end{lemma}
The proof is sketched in the appendix.
Let $f^{(i)}_{\varepsilon, u}(x) \triangleq (\chi^{(i)}_u\star\phi_\varepsilon)(x)$ be as in Lemmas \ref{main} and \ref{main2}. It follows by linearity that the covariance has the asymptotic form as $t\rightarrow \infty$
\begin{equation}\label{mylimitcovar}
{\bf Cov}\Bigl(S_t(f^{(i)}_{\varepsilon(t), u}), \, S_t(f^{(i)}_{\varepsilon(t), v})\Bigr) = \Re\int\limits_{-\log t/\lambda(t)}^{\log t/\lambda(t)} |w| \widehat{f^{(i)}_{\varepsilon(t), u}}(w) \overline{\widehat{f^{(i)}_{\varepsilon(t), v}}(w)}\,dw + o(1).
\end{equation}
It remains to show that this asymptotic 
is the same as that of the covariance in \eqref{covar} and \eqref{var} for $i=1$ and 
\eqref{covar2} and \eqref{var2} for $i=2.$ This is shown in the next lemma.
\begin{lemma}[Limit covariance]\label{limit}
Let $\lambda(t)$ satisfy \eqref{lt} and $\varepsilon(t)$ satisfy \eqref{vart} for $i=1$ or \eqref{vart2} for $i=2,$ respectively,
and the scalar product be as in \eqref{scalarf}. In the limit $t\rightarrow\infty,$
\begin{equation}
\Re\int\limits_{-\log t/\lambda(t)}^{\log t/\lambda(t)} |w| \widehat{f^{(i)}_{\varepsilon(t), u}}(w) \overline{\widehat{f^{(i)}_{\varepsilon(t), v}}(w)}\,dw = \langle f^{(i)}_{\varepsilon(t), u},\,f^{(i)}_{\varepsilon(t), v}\rangle
+ o(1).
\end{equation}
\end{lemma}
\begin{proof} {\bf of Lemma \ref{limit}.}
Given the equality between the representations of the scalar product in \eqref{scalarf} and \eqref{scalar}, we need to estimate
$\Re\int_{|w| >\log t/\lambda(t)}|w| \widehat{f^{(i)}_{\varepsilon(t), u}}(w) \overline{\widehat{f^{(i)}_{\varepsilon(t), v}}(w)}\,dw.$ Denote $\alpha \triangleq
\log t/\lambda(t),$ $f_\varepsilon(x)\triangleq f^{(i)}_{\varepsilon, u}(x),$ and $g_\varepsilon(x) \triangleq
f^{(i)}_{\varepsilon, v}(x).$ Then, by substituting the definition of the Fourier transform and integrating by parts,
\begin{equation}
\Re\int\limits_{-\alpha}^\alpha |w| \hat{f}_\varepsilon(w)\overline{\hat{g}_\varepsilon(w)}\,dw  = 
-\frac{1}{2\pi^2} \int\limits_{0}^\alpha \frac{dw}{w}  \Bigl[\iint f'_\varepsilon(x)\,g'_\varepsilon(y)  
\bigl(\cos(wy)-\cos(w(y-x))\bigr) dxdy\Bigr].
\end{equation}
The $\cos(wy)$ term can be replaced with $\cos(w)$ when integrating from 0 to
$\alpha$ and dropped when integrating from $\alpha$ to infinity 
because $\int f'_\varepsilon(x)\,dx=0.$ 
Using the Frullani integral in the form
\begin{equation}
\int\limits_{0}^\infty \frac{dw}{w} 
\bigl(\cos(w)-\cos(w(y-x))\bigr) = \log|y-x|,
\end{equation}
we wish to show that the remainder term satisfies the
estimate in the limit
$\alpha\rightarrow\infty$ and $\varepsilon\rightarrow 0,$
\begin{equation}\label{mainestimate}
\int\limits_{\alpha}^\infty \frac{dw}{w} \Bigl[
\iint f'_\varepsilon(x)\,g'_\varepsilon(y) 
\cos(w(y-x))dxdy\Bigr]  = O(1/\alpha\,\varepsilon).
\end{equation}
Let $i=1$ for concreteness. Recalling $f'_\varepsilon(x)=\phi_\varepsilon(x)-\phi_\varepsilon(x-u),$
$g'_\varepsilon(y)=\phi_\varepsilon(y)-\phi_\varepsilon(y-v),$ and the assumption that $\phi(x)$ is compactly supported so that
the only values of $x$ and $y$ that contribute to these integrals are concentrated in intervals of size $\varepsilon,$ it is sufficient to
show
\begin{align}
\int\limits_{\alpha}^\infty \frac{dw}{w} \Bigl[
\int \phi_\varepsilon(x)\,
\cos(w\,x)dx\Bigr]  = & \int\limits_{\alpha\varepsilon}^\infty \frac{dw}{w} \Bigl[\int \phi(x)\,
\cos(w\,x)dx\Bigr] \nonumber \\
= & O(1/\alpha\,\varepsilon),
\end{align}
which follows by integrating the inner integral by parts.
\qed
\end{proof}

\begin{remark}
One concludes that the statistic $S^{(1)}_t(\mu, u)$ and the centered gaussian
free field $\bar{\omega}_{\mu, \varepsilon(t)}(u)$ have asymptotically the same law,
as do the statistic $S^{(2)}_t(\mu, u)$ and the gaussian free field plus an independent gaussian,
as in \eqref{S1approxlaw} and \eqref{S2approxlaw} with the scale $\varepsilon(t).$ The difference between 
$S^{(1)}_t(\mu, u)$ and $S^{(2)}_t(\mu, u)$ on the one hand and the corresponding free fields on the other 
is that these statistics are only gaussian in the limit $t \rightarrow \infty,$ whereas the free fields are gaussian 
for $\varepsilon>0.$ We will assume that deviations of these statistics from their gaussian limits can be ignored 
for sufficiently large $t$ for the purpose of carrying out calculations. 
We can quantify this assumption in terms of the behavior of $Y_t^{(i)}$ near its gaussian limit in \eqref{limitofSt},
by conjecturing the asymptotic of $S^{(i)}_t(\mu, u)$ to be
\begin{equation}\label{theassumption}
S^{(i)}_t(\mu, u) \overset{{\rm in \,law}}{=} \sigma^{(i)}_t \mathcal{N}(0,\,1)+o(1),
\end{equation}
which means that $Y^{(i)}_t$ converges to $\mathcal{N}(0,\,1)$ faster (in the sense of variance) than 
$\sigma^{(i)}_t{}^2\propto -\log\varepsilon (t)$ goes to infinity. We can heuristically estimate the rate
of convergence of $Y^{(i)}_t$ to $\mathcal{N}(0,\,1)$ by the relative error in the formula for the
variance of our statistics. It follows from \eqref{mainestimate} in the proof of Lemma \ref{limit}
that the relative error is 
$O\bigl(\lambda(t)/\log t\,\varepsilon(t)\bigr)/|\log\varepsilon (t)|$
so that the rate of convergence $\ll 1/|\log\varepsilon (t)|$ by \eqref{vart} as required for \eqref{theassumption} to be true. 
Obviously, the assumption in \eqref{theassumption} is a major gap between the weak and strong conjectures. 
\end{remark}
\begin{proof} {\bf of Conjecture \ref{Mellin}.} 
The argument is the same as that for Conjecture \ref{Mellinvar} above except that we use Lemma \ref{modified} instead of Theorem \ref{strong} and
Lemma \ref{limit} instead of Lemmas \ref{main} and \ref{main2}, and, in addition, assume
that $S^{(i)}_t(\mu, u)$ is near-gaussian for large but finite $t$ as remarked above in \eqref{theassumption}.
\qed
\end{proof}

The proofs of the corollaries are given below. They are the same for
either the single or double limit on the left-hand side. For concreteness, they are stated for the single limit as in Section 2.
\begin{proof} {\bf of Corollary \ref{posmoments}.}
We computed the positive moments and verified \eqref{selberg2} 
in Lemma \ref{single}. On the other hand, using
\eqref{M2}, \eqref{mom2plus}
follows from the known formula for the positive integral moments of the Selberg integral distribution, cf. \eqref{posmomentformula}. The argument for \eqref{selberg1} and \eqref{mom1plus} is the same.
\qed
\end{proof}
\begin{proof} {\bf of Corollary \ref{negmoments}.}
This again follows from Conjecture \ref{Mellin} and the known formula for the negative integral moments of the Selberg integral distribution, cf. \eqref{negmomentformula}.
\qed
\end{proof}
\begin{proof} {\bf of Corollary \ref{jointintegralmom}.}
This is a special case of Lemma \ref{joint}.
\qed
\end{proof}
\begin{proof} {\bf of Corollary \ref{asymptotic}.}
\eqref{asymptotic2} is a special case of \eqref{asymptoticgeneral}
corresponding to $\lambda_1=\lambda_2=0.$ \eqref{asymptotic1} 
is a special case of the following asymptotic expansion in the limit $\mu\rightarrow 0,$
\begin{gather}
\mathfrak{M}(q\,|\,\mu,\mu\,z,0) \thicksim
\exp\Bigl(\sum\limits_{r=0}^\infty
\Bigl(\frac{\mu}{2}\Bigr)^{r+1}
\frac{1}{r+1}\Bigl[\zeta(r+1)\bigl[\frac{B_{r+2}(q+1)
+B_{r+2}(q)}{r+2} +\nonumber \\
+ \frac{B_{r+2}(q-2z)-B_{r+2}(-2z)-2B_{r+2}}{r+2} -q\bigr]  
+\bigl(\zeta(r+1)-1\bigr)\bigl[\frac{B_{r+2}(q-1-2z)-B_{r+2}(2q-1-2z)}{r+2}\bigr]\Bigr]\Bigr)
.\label{asymptoticz}
\end{gather}
This expansion coincides with the intermittency expansion of ${\bf E}\Bigl[\bigl(
\int_0^1 s^{\mu\,z} \,M_\mu(ds)\bigr)^q\Bigr]$ as one can verify by the methods
of the previous section. Another way of deriving \eqref{asymptoticz} is to
follow the argument that we used in \cite{MeIMRN} to derive \eqref{asymptoticgeneral} from
\eqref{Mfunc}. Then, \eqref{asymptotic1} corresponds to $z=q$ in \eqref{asymptoticz}.
\qed
\end{proof}
\begin{proof} {\bf of Corollary \ref{covstructure}.}
We will give the argument for the second statistic. We have by \eqref{S2approxlaw},
\begin{align}
{\bf Cov}\Bigl(\log \int\limits_{s_1}^{s_1+\Delta} e^{S^{(2)}_t(\mu,u)} du, \, \log \int\limits_{s_2}^{s_2+\Delta} e^{S^{(2)}_t(\mu,u)} du\Bigr) & \approx
{\bf Cov}\Bigl(\log \int\limits_{s_1}^{s_1+\Delta} e^{\omega_{\mu, \varepsilon(t)}(u)} du, \, \log \int\limits_{s_2}^{s_2+\Delta} e^{\omega_{\mu, \varepsilon(t)}(u)} du\Bigr)- \nonumber \\
& -\mu(3\log\varepsilon(t)-\kappa),
\end{align}
and the result follows by \eqref{covstrucmass} and the stationarity property of the limit lognormal measure.
The argument for the first statistic is similar but is more involved as it requires a generalization of  \eqref{covstrucmass}
for the centered gaussian free field, which follows from \eqref{jointmomentsformula} in the same way as \eqref{covstrucmass}
follows from \eqref{jointmomentlimitmeasure}. The details are straightforward and will be omitted.
\qed
\end{proof}
\begin{proof} {\bf of Corollary \ref{noncentral}.}
Given Conjecture \ref{Mellin}, these results follow from
Theorems \ref{BSM} and \ref{B2}.
\qed
\end{proof}
\begin{proof} {\bf of Corollary \ref{multifrac}.} 
Given the multifractal law of the limit lognormal measure, cf. \eqref{multifractallaw}, it is sufficient to
show the identity for any $0<s<1$
\begin{equation}
\lim\limits_{t\rightarrow\infty} e^{\mu(\log \varepsilon(t)-\kappa)\frac{q(q+1)}{2}} 
{\bf E} \Bigl[\Bigl(\int_0^s u^{-\mu q} e^{S^{(1)}_t(\mu,u)} du\Bigr)^q\Bigr]  = {\bf E} \Bigl[ \Bigl(\int_{0}^s M_\mu(du)\Bigr)^q\Bigr]. 
\end{equation}
which is a special case of \eqref{gen2} corresponding to $\varphi(u)=1$ and $I=[0, s]$ (by formally replacing  
$\bar{\omega}_{ \mu,\varepsilon}(u)$ with $S^{(1)}_t(\mu, u)$ as in the derivation
of Conjecture \ref{Mellin} above). 
\qed
\end{proof}
\begin{proof} {\bf of Corollary \ref{Multi}.} 
Given the multiscaling law of the limit lognormal measure, \eqref{multilaw2} follows from \eqref{multi}. To prove \eqref{multilaw1}, we need to generalize
\eqref{multifractallaw} to the following identity in law,
\begin{equation}\label{multifracq}
\int_0^s u^{\mu q}\,M_{\mu}(du) \overset{{\rm in \,law}}{=} s^{1+\mu q}\,e^{\Omega_s}\,\int_0^1 u^{\mu q}\,M_{\mu}(du),
\end{equation}
where $\Omega_s$ is as in \eqref{OmegaE} and \eqref{OmegaV},
which implies the multiscaling law
\begin{align}
{\bf E}\Bigl[\Bigl(\int_0^s u^{\mu q}\,M_{\mu}(du)\Bigr)^q\Bigr] & \propto 
{\bf E}\Bigl[\Bigl(s^{1+\mu q}\, e^{\Omega_s}\Bigr)^q\Bigr], \nonumber\\
& \propto s^{q+\frac{\mu}{2}(q^2+q)} 
\end{align}
as a function of $s<1.$ Finally, \eqref{multifracq} is a simple corollary 
of \eqref{invariance}. 
\qed
\end{proof}
\begin{remark}
It should be clear from Theorem \ref{BSM} and Theorem \ref{general} 
that one can obtain general explicit formulas for $\varphi(u)=u^{\lambda_1} (1-u)^{\lambda_2}.$ The reason that we restricted ourselves to $\lambda_1=\lambda_2=0$ in Conjectures \ref{Mellinvar} and \ref{Mellin} is simplicity.
The reader who is interested in the general case can easily find the desired formulas in Section 3. 
\end{remark}

\section{Conclusions}
\noindent We have formulated two versions of a precise conjecture on limits of rescaled Mellin-type transforms of the exponential functional of the Bourgade-Kuan-Rodgers statistic in the mesoscopic regime. The conjecture is based on our construction of  particular Bourgade-Kuan-Rodgers statistics of Riemann zeroes that converge to modifications of the centered gaussian or gaussian free fields. The statistics are defined by smoothing the indicator function of certain bounded or unbounded subintervals of the real line. The smoothing is effected by a rescaled bump function. In the weak version of the conjecture the asymptotic scale of the bump function $\varepsilon$ is fixed so that the resulting statistics $S^{(i)}_t(\mu, u, \varepsilon),$ $i=1,2,$
satisfy the Bourgade-Kuan-Rodgers theorem. In the strong version, this scale $\varepsilon(t)$ is chosen to be mesoscopic,
$\lambda(t)/\log t\ll\varepsilon(t)\ll1$ for $i=1$ and both $\lambda(t)/\log t\ll\varepsilon(t)\ll1$ and $\varepsilon(t)\gg 1/\lambda(t)$ for
$i=2,$  so that the statistics $S^{(i)}_t(\mu, u)$  satisfy an extension of the Bourgade-Kuan-Rodgers theorem that we formulated in the paper. We have computed the double limit over the scale $\varepsilon$ and $t$ in the weak case and the single limit over $t$ in the strong case of rescaled Mellin-type transforms of the exponential functional of both statistics as if the statistics were the centered gaussian free field or the gaussian free field plus an independent gaussian random variable, respectively. The exponential functional of the gaussian free field is an important object in mathematical physics known as the limit lognormal stochastic measure (or lognormal multiplicative chaos). By using a Girsanov-type result, we have found an appropriate rescaling factor to compute two Mellin-type transforms of the exponential functional of the centered field in terms of the Mellin transform of the exponential functional of the free field itself, \emph{i.e.} the Mellin transform of the total mass of the limit lognormal measure, resulting
in the conjecture for the first statistic. The conjecture for the second statistic follows directly from its convergence to
the gaussian free field plus an independent gaussian, the latter being responsible for rescaling. In both cases, the rescaling factors are determined by the asymptotic scale and the choice of the bump function that effect the smoothing. Finally, our conjectural knowledge of the distribution of the total mass of the measure has allowed us to calculate a number of quantities that are associated with the statistics exactly.

The principal difference between the weak and strong conjectures is their
informational content. Conjecture \ref{Mellinvar} alone can be thought of as a number theoretic re-formulation of Conjecture \ref{myresult} on the equality of  the Selberg integral distribution and the law of the total mass of the limit lognormal measure. It associates the Selberg integral distribution with the zeroes but does not contain any information about their distribution that is not already contained in the Bourgade-Kuan-Rodgers theorem. Indeed, the order of the $u$ integral and the $t$ limit can be interchanged at a finite $\varepsilon>0,$ resulting in the Mellin transform of the exponential functional of a gaussian field, which converges to the centered gaussian free field when $i=1$ or the sum of the gaussian free field plus an independent gaussian random variable when $i=2$ as $\varepsilon\rightarrow 0,$ so that the weak conjecture only requires the $t\rightarrow \infty$ limit of the $S^{(i)}_t(\mu,u,\varepsilon)$ statistic. 
On the other hand, Conjecture \ref{Mellin} combines the $\varepsilon \rightarrow 0$ and $t\rightarrow \infty$ limits into a single limit so that the order of the $u$ integral and the resulting limit can no longer be interchanged because the statistic $S^{(i)}_t(\mu,u),$ unlike $S^{(i)}_t(\mu,u,\varepsilon),$ converges to the centered gaussian free field when $i=1$ or the gaussian free field plus an independent gaussian when $i=2$ at our mesoscopic scale $\varepsilon(t)$ and so becomes singular as $t\rightarrow \infty.$ The strong conjecture assumes that deviations of $S^{(i)}_t(\mu,u)$ from its gaussian limit are negligible 
at large but finite $t,$ thereby providing some new information about the statistical distribution of the zeroes at finite $t.$ In particular, as the strong conjecture
fits into the framework of mod-gaussian convergence, the normality zone and precise deviations of tails of our exponential functionals can be computed from
the general theory and explicit knowledge of our limiting functions.  

We have provided a self-contained review of some of the key properties of the limit lognormal measure and the distribution of its total mass to make our work
accessible to a wider audience. In particular, we have covered the invariances of the gaussian free field, the multifractal law of the limit measure, the derivation of the law of its total mass by exact renormalization, and several characterizations of the Selberg integral distribution, which is believed to describe  
the law of the total mass. The Selberg integral distribution is a highly non-trivial, log-infinitely divisible probability distribution having the property that its positive integral moments are given by the Selberg integral of the same dimension as the order of the moment. We have reviewed both its analytic and probabilistic structures that are relevant to out calculations. 

We have provided a number of calculations that support our conjecture. 
Our calculations are universal as they apply to any asymptotically gaussian linear statistic having the covariance structure 
that is given by the Bourgade-Kuan-Rodgers formula. In particular, they apply to the GUE statistics of Fyodorov \emph{et. al.} \cite{FKS} and
provide a theoretical explanation, \emph{via} convergence to the gaussian free field, for why our results for the Bourgade-Kuan-Rodgers and GUE statistics are the same.
Our arguments are however not mathematically rigorous as we do not know how our statistics behave near their gaussian limits. Our assumption that their behavior in the limit can be used to do calculations near the limit in the strong case is the principal mathematical gap between the weak and strong conjectures that renders our use of the free fields in place of the statistics heuristic. We have quantified that the variances of 
our statistics should converge to their asymptotic limits
faster than $1/|\log\varepsilon(t)|$ for this assumption to be valid and explained heuristically why we expect this to be true.

In broad terms, on the one hand, our conjecture relates a limit of a statistic of Riemann zeroes with the Selberg integral and, more generally, the Selberg integral probability distribution and so associates a non-trivial, log-infinitely divisible distribution with the zeroes. On the other hand,
our conjecture implies that the limit lognormal measure can be modeled in terms
of the zeroes. As this measure appears naturally in various contexts that involve multifractality, we can speculate that there is a number theoretic interpretation of multifractal phenomena. In particular, a proof of (even the weak) the conjecture might lead to a number theoretic proof of the conjectured equality of the law of the total mass of the limit lognormal measure and the Selberg integral distribution.

We have interpreted our statistics as fluctuations of the smoothed error term in the zero counting function so that
our conjecture gives rescaled Mellin transforms of the exponential functional of these fluctuations. Aside from verification or disproof
of our conjecture, it would be quite interesting to see what properties of the error term follow from the conjecture and to
compute the rate of convergence of our statistics to their gaussian limits. We believe that our conjecture is only valid for
$1\ll\lambda(t)\ll \log t$ and breaks down for $\lambda(t)\thicksim 1$ due to the expected presence of yet-to-be-determined
arithmetic corrections. Finally, the appearance of the Selberg integral distribution 
in the work of Fyodorov and Keating \cite{YK} and in our paper begs the question of formulating a general
statement about logarithmic correlations in the statistical value distribution of $\log\zeta,$ which we believe will clarify 
the relationship between their conjecture and ours.

\section*{Acknowledgements}
The author wants to thank Jeffrey Kuan for a helpful correspondence relating to ref. \cite{BK}.
The author also wishes to express gratitude to Nickolas Simm for bringing ref. \cite{FKS} to our attention, pointing out
the analogy between the Bourgade-Kuan-Rodgers and GUE statistics, and simplifying our proof of Eq. (16). 
The author is grateful to the referees for many helpful suggestions.

\appendix
\section{Appendix}
\numberwithin{equation}{section} \setcounter{equation}{0} \noindent
In this section we will give a proof\footnote{This proof is due to Nickolas Simm. We originally established this result as a corollary of Eq. (17) in \cite{BK}. Upon 
seeing our derivation, Nickolas Simm found a much simpler proof, which is given here.} of \eqref{resultS} and a sketch of the proof of Lemma \ref{modified}.
The starting point is the second equation following Eq. (16) in \cite{BK},
\begin{equation}\label{secBK}
\sum\limits_\gamma \Im\Bigl(\frac{1}{\gamma(t)-(x+i y)}\Bigr) -\frac{1}{2}\frac{\log t}{\lambda(t)} = \frac{1}{\lambda(t)}\Re \frac{\zeta'}{\zeta}\Bigl(
\frac{1}{2}+\frac{y}{\lambda(t)} + i(\omega t+\frac{x}{\lambda(t)}\Bigr) + \frac{1}{\lambda(t)}\,O\Bigl(|\log(\omega+\frac{x}{t\lambda(t)}|\Bigr).
\end{equation}
Recalling the identity
\begin{equation}
\lim\limits_{y\rightarrow 0} \int f(x) \Im\Bigl(\frac{1}{\gamma-(x+iy)}\Bigr)\,dx = \pi \,f(\gamma),
\end{equation}
we multiply \eqref{secBK} by $f(x)/\pi,$ integrate with respect to $x,$ and let $y\rightarrow 0.$ We then obtain by \eqref{St}
\begin{align}
S_t(f) & = \frac{1}{\pi \lambda_t} \int\limits_\mathcal{\mathbb{R}} f(x)\ \,\Re \frac{\partial}{\partial s} \log\zeta(s)\vert_{s=\frac{1}{2}+
i\bigl(\omega t + \frac{x}{\lambda_t}\bigr)}\,dx + O\bigl(1/\lambda(t)\bigr), \\
& = \frac{1}{\pi} \int\limits_\mathcal{\mathbb{R}} f(x) \,\frac{\partial}{\partial x}\Im \Bigl[ \log\zeta\Bigl(\frac{1}{2}+
i\bigl(\omega t + \frac{x}{\lambda_t}\bigr)\Bigr)\Bigr]\,dx + O\bigl(1/\lambda(t)\bigr), \\
& = -\int\limits_\mathcal{\mathbb{R}} f'(x)\ \,S\Bigl(
\omega t + \frac{x}{\lambda_t}\Bigr)\,dx + O\bigl(1/\lambda(t)\bigr)
.
\end{align} 
It should be noted that $\Re(\zeta'/\zeta)(s)$ is integrable along the critical line due to Hadamard's factorization, which shows that the singular part
of $(\zeta'/\zeta)(s)$ near the zeroes is of the form $1/(s-\gamma)$ and so is purely imaginary along the critical line, hence vanishing upon taking
the real part. It is clear from \eqref{secBK} that the error term is indeed $O\bigl(1/\lambda(t)\bigr)$ provided $f(x)$ is compactly supported and the support is
independent of $t.$ If the support depends on $t$ as in \eqref{Sut}, provided its size$\ll t\lambda(t),$ the error is $O\bigl(||f||_1/\lambda(t)\bigr)$
as indicated in footnote 7.

The proof of Lemma \ref{modified} follows verbatim the proof of Theorem \ref{divergvar} given in
\cite{BK}. We will only indicate the main steps here and refer the reader to
\cite{BK} for all details, including definitions of special number theoretic functions that are needed in the proof.
\begin{proof} {\bf of Lemma \ref{modified}.} Let $\chi_u(x)=\chi_{[0,\,u]}(x),$ \emph{i.e.}
$i=1.$
Let $\{c_k\}$ and $\{u_k\}$ be fixed and denote 
\begin{equation}\label{ft}
f_t(x) \triangleq \Bigl(\sum c_k \chi_{u_k}\star\phi_{\varepsilon(t)}\Bigr)(x).
\end{equation}
Then, $f_t$ and $f_t'$ are bounded in $L^1$ uniformly in $t,$ 
and $f_t''$ satisfies
\begin{equation}
||f_t''||_1 = O(1/\varepsilon(t)).
\end{equation}
The Fourier transform of $f_t$ 
satisfies uniformly in $t$ as $w\rightarrow \infty$
\begin{equation}\label{fourierbounds}
w\,|\widehat{f_t}(w)|^2, \bigl(w\,|\widehat{f_t}(w)|^2\bigr)' = O(1/w).
\end{equation}
Then, by the approximate explicit formula of Bourgade and Kuan, cf. Proposition 3 in \cite{BK}, 
\begin{equation}\label{explicit}
S_t(f_t) = \frac{1}{\lambda(t)} \sum\limits_{n\geq 1} \frac{\Lambda_{\sqrt{t}}(n)}{\sqrt{n}}
\Bigl(\widehat{f_t}\Bigl(\frac{\log n}{\lambda(t)}\Bigr) n^{i\omega t} +
\text{cc}\Bigr) + \text{error term},
\end{equation}
where $\Lambda_{u}(n)$ denotes Selberg's smoothed von Mangoldt function.
In our case, as in the case of Theorem \ref{divergvar}, the error term is of the 
order $O(\lambda(t) /\log t \,\varepsilon(t)).$
Let $\sigma_t$ be as in \eqref{mylimitvar} and denote 
\begin{equation}
Y_t \triangleq \frac{\sqrt{2}}{\sigma_t\lambda(t)} \sum\limits_{\text{primes}\; p} \frac{\Lambda_{\sqrt{t}}(p)}{\sqrt{p}}
\widehat{f_t}\Bigl(\frac{\log p}{\lambda(t)}\Bigr) p^{i\omega t}.
\end{equation}
Then, by \eqref{explicit} and Lemma 4 of \cite{BK}, 
\begin{equation}\label{Stft}
S_t(f_t) = \frac{\sigma_t}{\sqrt{2}} \Bigl(Y_t+\overline{Y_t}\Bigr) + O\Bigl(\frac{\lambda(t)}{\log t\,\varepsilon(t)}\Bigr) + o(1)
\end{equation} 
in the limit $t\rightarrow \infty.$
Finally, by Lemma 5 and Proposition 6 of \cite{BK}, $Y_t$ converges to the standard complex normal variable
\begin{equation}
Y_t\rightarrow \frac{\mathcal{N}(0, 1)+ i\mathcal{N}(0, 1)}{\sqrt{2}},
\end{equation}
The result for $i=1$ follows.

Now, let $\chi_u(x)=\chi_{[-1/\varepsilon(t),\,u]}(x)$ and $f_t$ be as in \eqref{ft}. The argument given above still goes through 
but requires somewhat more delicate estimates. We have for $n=0, 1, 2,$
\begin{align}
||f^{(n)}_t||_1 & = O(1/\varepsilon(t)), \label{l1} \\
||x\,\log x\,f^{(n)}_t||_1 & = O(\log(1/\varepsilon(t))/\varepsilon^2_t),
\end{align}
so that \eqref{explicit} still holds with the error term of the 
order $O(\lambda(t) /\log t \,\varepsilon(t)).$ To verify \eqref{Stft} we need Lemma 4 of \cite{BK}, which requires
that $||f_t||_1$ be uniformly bounded, whereas we have \eqref{l1} instead. However, a careful reading of the proof of
Lemma 4 indicates that it still holds provided
\begin{equation}
\frac{||f_t||_1}{\lambda(t)}=o(1),
\end{equation}
which in our case translates into
\begin{equation}
\frac{1}{\varepsilon(t)\,\lambda(t)}=o(1).
\end{equation}
Thus, \eqref{Stft} holds given the condition in \eqref{vart2}. Similarly, Lemma 5 of \cite{BK} requires 
the bounds in \eqref{fourierbounds}, whereas we have instead
\begin{align}
w|\widehat{f_t}(w)|^2 & = O(1/w), \\
\bigl(w|\widehat{f_t}(w)|^2\bigr)' & = O(1/w\,\varepsilon(t)). \label{four2}
\end{align}
Once again, a careful reading of the proof of Lemma 5 indicates that the bound in \eqref{four2} is sufficient provided
$1/\varepsilon(t)\,\lambda(t)=o(1)$ as in \eqref{vart2}. The rest of the argument goes though verbatim.
\qed 
\end{proof}


\begin{thebibliography}{50}

\bibitem{RV2} R. Allez, R. Rhodes, V. Vargas (2013), Lognormal $\star-$scale invariant random measures, {\it Probab. Theory Relat. Fields} {\bf 155}, 751-788.

\bibitem{Jones} K. Astala, P. Jones, A. Kupiainen, E. Saksman
(2011), Random conformal weldings, {\it Acta Mathematica} {\bf 207}, 203-254.

\bibitem{MRW} E. Bacry, J. Delour, J.-F. Muzy (2001a),
Multifractal random walk, {\it Phys. Rev. E} {\bf 64}, 026103.

\bibitem{BDM} E. Bacry, J. Delour, J.-F. Muzy (2001b), Modelling financial time series using multifractal random walks, {\it Physica A} {\bf 299}, 84-92.

\bibitem{BM1} E. Bacry, J.-F. Muzy (2003), Log-infinitely divisible multifractal random walks, {\it Comm. Math. Phys.} {\bf 236}, 449-475.

\bibitem{Genesis} E. W. Barnes (1899), The genesis of the double gamma functions, \emph{Proc. London Math. Soc.} \textbf{s1-31},
358-381.

\bibitem{mBarnes} E. W. Barnes (1904), On the theory of the multiple gamma function. \emph{Trans. Camb. Philos. Soc.} \textbf{19}, 374-425.

\bibitem{barraletal} J. Barral, A. Kupiainen, M. Nikula, E. Saksman, C. Webb (2013), Basic properties of critical lognormal multiplicative chaos, preprint,
http://arxiv.org/abs/1303.4548.

\bibitem{Pulses} J. Barral, B. B. Mandelbrot (2002), Multifractal products of
cylindrical pulses, {\it Probab. Theory Relat. Fields} {\bf 124},
409-430.

\bibitem{BenSch} I. Benjamini, O. Schramm (2009), KPZ in one dimensional random geometry of multiplicative cascades, {\it Comm. Math. Phys.} {\bf 289}, 653-662.

\bibitem{Berry} M. V. Berry (1988), Semiclassical formula for the number variance of the Riemann zeros, {\it Nonlinearity} {\bf 1}, 399-407.

\bibitem{BoKe} E.B. Bogomolny and J.P. Keating (1996), Gutzwiller's trace formula and spectral statistics: beyond the diagonal approximation, {\it Phys. Rev. Lett.}
{\bf 77}, 1472-1475.

\bibitem{BoKe2} E.B. Bogomolny and J.P. Keating (2013), A method for calculating spectral statistics based on random-matrix
universality with an application to the three-point correlations of the Riemann zeros, {\it J. Phys. A} {\bf 46},
305203.

\bibitem{Bourgade10} P. Bourgade (2010), Mesoscopic fluctuations of the zeta zeros, {\it Probab. Theory Relat. Fields} {\bf 148}, 479-500.

\bibitem{BK} P. Bourgade and J. Kuan (2013), Strong Szeg\H{o} asymptotics and zeros of the zeta function, {\it Comm. Pure Appl. Math} {\bf 67}, 1028-1044,
http://arxiv.org/abs/1203.5328 (corrected version).

\bibitem{Conreyetal} J.B. Conrey, D.W. Farmer, and M.R. Zirnbauer (2008), Autocorrelation of ratios of $L$-functions,
{\it Comm. Number Theory and Physics}, {\bf 2}, 593-636.

\bibitem{CS} J.B. Conrey, N.C. Snaith (2008), Correlations of eigenvalues and Riemann zeros, {\it Comm. Number Theory and Physics}, {\bf 2}, 477-536.

\bibitem{dupluntieratal} B. Duplantier, R. Rhodes, S. Sheffield, V. Vargas (2014), Critical gaussian multiplicative chaos: convergence of the derivative martingale, {\it Ann. Probab.} {\bf 42}, 1769-1808.

\bibitem{DS} B. Duplantier, S. Sheffield (2011), Liouville quantum gravity and KPZ, {\it Invent. Math. } {\bf 185}, 333-393.

\bibitem{Fetal} D. W. Farmer, S. M. Gonek, and C. P. Hughes (2007), The maximum size of $L$-functions, {\it J. Reine Angew. Math.} {\bf 609}, 215-236.

\bibitem{Feray} V. F\'eray, P.-L. M\'eliot, A. Nikeghbali (2015), Mod-phi convergence and precise deviations, http://arxiv.org/abs/1304.2934.

\bibitem{ForresterBook} P. J. Forrester (2010),  Log-Gases and Random
Matrices, \emph{Princeton University Press}, Princeton.

\bibitem{Fujii} A. Fujii (1974), On the zeros of Dirichlet L-functions I, {\it Trans. Amer. Math. Soc.} {\bf 196}, 225-235.

\bibitem{Fujii2} A. Fujii (1997), On the Berry conjecture. {\it J. Math. Kyoto Univ.} {\bf 37}, 55-98.

\bibitem{FyoBou} Y. V. Fyodorov and J. P. Bouchaud (2008), Freezing and extreme-value statistics in a random energy model with logarithmically correlated potential. \emph{J. Phys. A, Math Theor.} \textbf{41}, 372001.

\bibitem{YO} Y. V. Fyodorov, O. Giraud (2015), High values of disorder-generated multifractals and logarithmically correlated processes, {\it Chaos, Solitons \& Fractals}, {\bf 74}, 15-26.

\bibitem{FKS} Y. V. Fyodorov, B. A. Khoruzhenko, and N. J. Simm (2015), Fractional Brownian motion with Hurst index $H=0$
and the Gaussian Unitary Ensemble, http://arxiv.org/abs/1312.0212.

\bibitem{FLDR} Y. V. Fyodorov, P. Le Doussal, A. Rosso (2009), Statistical mechanics of logarithmic REM: duality, freezing and extreme value statistics of $1/f$ noises generated by gaussian free fields, \emph{ J. Stat. Mech. Theory Exp.}, P10005.

\bibitem{FLDR2} Y. V. Fyodorov, P. Le Doussal, A. Rosso (2012), Counting function fluctuations and extreme value threshold in multifractal patterns: the case study of an ideal $1/f$ noise, \emph{ J. Stat. Phys.} {\bf 149}, 898-920.

\bibitem{YK} Y. V. Fyodorov and J. P. Keating (2014), Freezing transitions and extreme values: random matrix theory, $\zeta(1/2+it),$ and disordered landscapes. \emph{ Philos. Trans. R. Soc. Lond. Ser. A Math. Phys. Eng. Sci.} \textbf{372}, 20120503. 

\bibitem{Hughesetal} C. P. Hughes, J. P. Keating, and N. O'Connell (2001), On the characteristic polynomial of a random unitary matrix. {\it Commun. Math. Physics}, {\bf 220}, 429-451.

\bibitem{HR} C. P. Hughes and Z. Rudnick  (2002), Linear statistics for zeros of Riemann's zeta function, {\it C. R.Math. Acad. Sci. Paris} {\bf 335},
667-670.

\bibitem{Jacod} J. Jacod, E. Kowalski, A. Nikeghbali (2011), 
Mod-gaussian convergence: new limit theorems in probability and
number theory. \emph{Forum Math.} \textbf{23}, 835-873.

\bibitem{K2} J.-P. Kahane (1987), Positive martingales and random measures, {\it  Chinese Ann. Math. Ser. B} {\bf 8}: 1-12.

\bibitem{Kar} V. Kargin (2013), On Fluctuations of Riemann's Zeta Zeros, {\it Probab. Theory Relat. Fields}, {\bf 157}, 575-604.

\bibitem{KS} J.P. Keating, N.C. Snaith (2000), Random matrix theory and
$\zeta(1/2+it),$ {\it Comm. Math. Phys.} {\bf 214}, 57-89.

\bibitem{secondface} B. B. Mandelbrot (1972), Possible refinement of the log-normal hypothesis concerning the distribution of energy dissipation in intermittent turbulence, in {\it Statistical Models and Turbulence}, M.
Rosenblatt and C. Van Atta, eds., Lecture Notes in Physics {\bf 12},
Springer, New York, p. 333.

\bibitem{Lan} B. B. Mandelbrot (1990), Limit lognormal multifractal measures, in {\it Frontiers of Physics: Landau Memorial Conference,} E. A. Gotsman {\it et al}, eds., Pergamon, New York, p. 309.

\bibitem{Meliot} P.-L. M\'eliot, A. Nikeghbali (2014), Mod-Gaussian convergence and its applications for models of statistical mechanics, http://arxiv.org/abs/1409.2849.

\bibitem{Mont} H.L. Montgomery (1973), The pair correlation of the zeta function, {\it Proc. Symp. Pure Math.} {\bf 24}, 181-193.

\bibitem{BM} J.-F. Muzy, E. Bacry (2002), Multifractal stationary random measures and multifractal
random walks with log-infinitely divisible scaling laws, {\it Phys.
Rev. E} {\bf 66}, 056121.

\bibitem{NikYor} A. Nikeghbali and M. Yor (2009), The Barnes G function
and its relations with sums and products of generalized gamma
convolutions variables. \emph{Elect. Comm. in Prob.} \textbf{14},
396-411.

\bibitem{Me2} D. Ostrovsky (2007), Functional Feynman-Kac equations for limit lognormal multifractals, {\it  J. Stat. Phys.} {\bf 127},
935-965.

\bibitem{Me3} D. Ostrovsky (2008), Intermittency expansions for limit lognormal
multifractals, {\it  Lett. Math. Phys.} {\bf 83}, 265-
280. 

\bibitem{Me4} D. Ostrovsky (2009), Mellin transform of the limit
lognormal distribution. \emph{Comm. Math. Phys.} \textbf{288}, 287-310.

\bibitem{Me5} D. Ostrovsky (2010), On the limit lognormal and other
limit log-infinitely divisible laws, {\it J. Stat. Phys.} {\bf 138},
890-911.

\bibitem{Me6} D. Ostrovsky (2011), On the stochastic dependence
structure of the limit lognormal process, {\it Rev. Math. Phys.}
{\bf 23}, 127-154.

\bibitem{MeIMRN} D. Ostrovsky (2013), Selberg integral as a meromorphic
function. \emph{Int. Math. Res. Not. IMRN}, \textbf{17}, 3988-4028.

\bibitem{Me13} D. Ostrovsky (2013), Theory of Barnes beta distributions. \emph{Elect. Comm. in Prob.} \textbf{18}, no. 59,  1-16.

\bibitem{Me14} D. Ostrovsky (2014), On Barnes beta distributions, Selberg integral and Riemann xi. \emph{Forum Math.}, DOI: 10.1515/forum-2013-0149.

\bibitem{RajRos}  B. S. Rajput, J. Rosinski (1989), Spectral representations of infinitely divisible processes, {\it Probab. Theory Relat. Fields} {\bf 82},  451-487.

\bibitem{RV1} R. Rhodes, V. Vargas (2011), KPZ formula for log-infinitely divisible multifractal random measures, {\it ESAIM: Probability and Statistics}     {\bf 15}, 358-371. 

\bibitem{Rodg} B. Rodgers (2014), A central limit theorem for the zeroes of the zeta function. {\it Int. J. Number Theory}, {\bf 10}, 483-511.

\bibitem{Ruij} S. N. M. Ruijsenaars (2000), On Barnes' multiple zeta and
gamma functions. \emph{Adv. Math.} \textbf{156}, 107-132.

\bibitem{Schmitt} F. Schmitt, D. Marsan (2001), Stochastic equations generating continuous multiplicative cascades, {\it Eur. J. Phys. B} {\bf 20}: 3-6.

\bibitem{Selberg} A. Selberg (1944), Remarks on a multiple integral, {\it Norske Mat. Tidsskr.} {\bf 26}: 71-78.

\bibitem{Selberg44} A. Selberg (1944), On the remainder in the formula for $N(T),$ the number of zeros of $\zeta(s)$ in the strip $0<t<T,$ {\it Avh. Norske Vid. Akad. Oslo. I.} , 27.

\bibitem{Selberg46} A. Selberg (1946), Contributions to the theory of the Riemann zeta-function, {\it  Arch. Math. Naturvid.} {\bf 48}, 89-155.

\bibitem{Titchmarsh} E. C. Titchmarsh (1986), The Theory of the Riemann Zeta-function, 2nd ed. \emph{Clarendon Press}, Oxford.


\end{thebibliography}
\end{document}